\documentclass[final,3p,times]{elsarticle}

\usepackage{amssymb}
\usepackage{graphicx,setspace}
\usepackage{psfrag}
\usepackage{amsfonts}
\expandafter\let\csname equation*\endcsname\relax
\expandafter\let\csname endequation*\endcsname\relax
\usepackage{amsmath}
\usepackage{amssymb, amsthm}
\usepackage{mathrsfs}
\usepackage{epstopdf}
\usepackage{float}
\usepackage{lineno,hyperref}
\usepackage{color}
\usepackage{subfigure}

\usepackage{amsmath}
\usepackage{dsfont} 
\usepackage{amssymb} 
\usepackage{color}
\usepackage{extarrows}
\usepackage{graphicx}
\usepackage{epstopdf}

\usepackage{todonotes}

\begin{document}
\newtheorem{definition}{Definition}[section]
\newtheorem{lemma}{Lemma}[section]
\newtheorem{remark}{Remark}[section]
\newtheorem{theorem}{Theorem}[section]
\newtheorem{proposition}{Proposition}
\newtheorem{assumption}{Assumption}[section]
\newtheorem{example}{Example}
\newtheorem{corollary}{Corollary}[section]
\def\ep{\varepsilon}
\def\Rn{\mathbb{R}^{n}}
\def\Rm{\mathbb{R}^{m}}
\def\E{\mathbb{E}}
\def\hte{\hat\theta}
\renewcommand{\theequation}{\thesection.\arabic{equation}}
\begin{frontmatter}

\title{Well-posedness and averaging principle for L\'evy-type McKean-Vlasov stochastic differential equations under local Lipschitz conditions}

\author{Ying Chao\fnref{addr1}}\ead{yingchao1993@xjtu.edu.cn}
\author{Jinqiao Duan\fnref{addr2,addr3}}\ead{duan@gbu.edu.cn}
\author{Ting Gao\fnref{addr2}}\ead{tgao0716@hust.edu.cn}
\author{Pingyuan Wei\corref{cor1}\fnref{addr4}}
\ead{weipingyuan@pku.edu.cn}\cortext[cor1]{Corresponding author}

\address[addr1]{School of Mathematics and Statistics, Xi'an Jiaotong University, Xi'an, Shaanxi 710049, China}
\address[addr2]{School of Mathematics and Statistics $\&$ Center for Mathematical Sciences, Huazhong University of Science and Technology, Wuhan, Hubei 430074, China}
\address[addr3]{Department of Mathematics $\&$ Department of Physics, Great Bay University, Dongguan, Guangdong 523000, China}
\address[addr4]{Beijing International Center for Mathematical Research, Peking University, Beijing 100871, China}

\begin{abstract}
In this paper, we investigate a class of McKean-Vlasov stochastic differential equations under L\'evy-type perturbations. We first establish the existence and uniqueness theorem for solutions of the McKean-Vlasov stochastic differential equations by utilizing the Euler-like approximation. Then under some suitable conditions, we show that the solutions of McKean-Vlasov stochastic differential equations can be approximated by the solutions of the associated averaged McKean-Vlasov stochastic differential equations in the sense of mean square convergence. In contrast to the existing work, a novel feature is the use of a much weaker condition---local Lipschitzian in the state variables, allowing for possibly super-linearly growing drift, but linearly growing diffusion and jump coefficients. Therefore, our results are suitable for a wider class of McKean-Vlasov stochastic differential equations. 
 
 \end{abstract}


\end{frontmatter}

\renewcommand{\theequation}{\thesection.\arabic{equation}}
\setcounter{equation}{0}

\section{Introduction}

McKean-Vlasov stochastic differential equations (SDEs) have received a great deal of attention in recent years, due to their wide applications across many fields such as stochastic controls, stochastic games and statistical physics. Equations of this type were first initiated in \cite{Mckean1966} inspired by the kinetic theory of Kac \cite{Kac1956}, and differ from usual SDEs in the sense that the coefficients additionally depend on the probability distribution of the solution process. In the literature, McKean-Vlasov SDEs are also referred to as mean-field SDEs. The reason is that these equations are obtained as the limits of some weakly interacting particles equations as the number of particles tends to infinity (so-called propagation of chaos). 

In view of the development on the aforementioned McKean-Vlasov SDEs, the noise processes considered are mainly Gaussian ones. However, the systems of practical relevance in physics and biology, sometimes have to be modeled by non-Gaussian noise. This can be verified by some abrupt jump in the individual particles and the related whole population. To reproduce the performance of these natural phenomena, it is appropriate to consider (non-Gaussian) L\'evy-type perturbations \cite{Applebaum2009,Duan2015,Liu2022}. In this paper, we focus on the following  $d$-dimensional L\'evy-type McKean-Vlasov SDE
\begin{align}\label{Main}
&dX_{\varepsilon}(t)=b\left(\frac{t}{\varepsilon},X_{\varepsilon}(t),\mathscr{L}_{X_{\varepsilon}(t)}\right) dt+\sigma\left(\frac{t}{\varepsilon},X_{\varepsilon}(t),\mathscr{L}_{X_{\varepsilon}(t)}\right)dW(t)+\int_{U}h\left(\frac{t}{\varepsilon},X_{\varepsilon}(t),\mathscr{L}_{X_{\varepsilon}(t)},z\right)\tilde{N}(dt,dz),\;\; X_{\varepsilon}(0)=x_0,
\end{align}
on $t\in [0,T]$ with a small parameter $\varepsilon>0$. Where $\mathscr{L}_{X(t)}$ denotes the law of $X(t)$ at time $t$, $W(t)$ is an $m$-dimensional standard Wiener process defined on the complete probability space $(\Omega, \mathcal{F}, (\mathcal{F}_t)_{t\geq0}, \mathbb{P})$ with $(\mathcal{F}_t)_{t\geq0}$ satisfying the usual conditions,  $N(dt,dz)$ is the Poisson random measure on a $\sigma$-finite measure space $(U,\mathcal{U},\nu)$ with $U\subseteq {\mathbb{R}^d}\backslash \{ 0\}$, independent of $W(t)$, and $\tilde N(dt,dz) = N(dt,dz) - \nu (dz)dt$ is its compensated Poisson random measure. The precise assumptions on the coefficients $b:[0,T]\times\mathbb{R}^d\times M_2(\mathbb{R}^d)\to\mathbb{R}^d$, $\sigma:[0,T]\times\mathbb{R}^d\times M_2(\mathbb{R}^d)\to\mathbb{R}^{d\times m}$, and $h:[0,T]\times\mathbb{R}^d\times M_2(\mathbb{R}^d)\times\mathbb{R}^d\to\mathbb{R}^d$ will be specified in later sections (see Section \ref{sec2} for the definition of $M_2(\mathbb{R}^d)$).
The first aim of this paper is to consider the well-posedness of the McKean-Vlasov SDEs in the form of  (\ref{Main}). Let us briefly review some previous works about the well-posedness for McKean-Vlasov SDEs with Brownian noise. Under the global Lipschitz condition, the existence and uniqueness of the strong solutions for McKean-Vlasov SDEs were obtained by using the fixed point theorem, for example, in \cite{Carmona2018probabilistic, Bahlali2020}. The results of the case with one-sided global Lipschitz drift term and global Lipschitz diffusion term can be found in \cite{Wang2018,Dos2022simulation}. 
To deal with the situation of locally Lipschitz with respect to (w.r.t. for short) the measure and globally Lipschitz w.r.t. the state variable, Kloeden and Lorenz \cite{Kloeden2010} established a method of constructing interpolated Euler-like approximations. Recently, an extension to local Lipschitz conditions w.r.t. the state variable but under uniform linear growth assumption was treated by Li \emph{et al.} \cite{Li2022strong}; see also \cite{Ding2021}.  Moreover, Hong \emph{et al.} \cite{Hong2022mckean} discussed the strong and weak well-posedness of a class of McKean-Vlasov SDEs with drift and diffusion coefficients fulfilling some locally monotone conditions, whereas they need to impose some extra structure on the coefficients to obtain a unique solution.

Unlike the case of Brownian noise, the related study of McKean-Vlasov SDEs with L\'evy noise is still in its infancy, but some interesting works are emerging. For example, Hao \emph{et al.} \cite{Hao2016} considered a class of L\'evy-type McKean-Vlasov SDEs satisfying global Lipschitz plus linear growth conditions, established the existence and uniqueness of solutions and studied the intrinsic link with nonlocal Fokker-Planck equations. The well-posedness results have been further developed in the case of super-linearly drift, diffusion and jump coefficients by using the fixed point theorem \cite{Mehri2020,Neelima2020}.

Motivated by the previous works on Brownian case as well as L\'evy case, in this paper, we aim to treat (\ref{Main}) only imposing local Lipschitz conditions w.r.t. the state variable, allowing for possibly super-linearly growing drift. We highlight that some essential difficulties occur. One the one hand, compared with the classical SDEs, the standard localization argument etc. cannot be applied directly due to the distribution dependent coefficients. On the other hand, the non-Gaussian Lévy noise gives rise to several difficulties both in analytic and probabilistic aspects. Therefore, the results of classical SDEs (even with L\'evy noise) or that of McKean-Vlasov SDEs with Brownian noise cannot be extended to McKean-Vlasov SDEs with L\'evy noise directly. In this paper, we develop a L\'evy-type technique of Euler-like approximations to overcome difficulties cased by the local condition and the distribution dependency. The crux of our method, which is different from the Brownian case \cite{Kloeden2010}, lies in dealing with the drift terms in more general conditions as well as the jump terms.

Apart from the existence and uniqueness of solutions, we are further interested in establishing a stochastic averaging principle for (\ref{Main}) with drifts of polynomial growth under local Lipschitz conditions w.r.t. the state variable. In fact, the averaging principle is a powerful method to extract effective dynamics from complex systems arising from mechanics, mathematics, and other research areas. Since the pioneer work of Khasminskii \cite{Khasminskii1968}, the averaging principle for usual SDEs has received a lot of attention, and has stimulated much of the study on controls, stability analyses, optimization methods. Although the problems considered take different forms (usually classified in terms of the noise or the conditions satisfied by its nonlinear terms), the essence behind the averaging method is to give a simplification of dynamical systems and obtain approximating solutions to differential equations, see, e.g., \cite{Xu2011, Ma2019, Pei2020}. Based on the idea of stochastic averaging, the second goal of this paper is to show that the solution of (\ref{Main}), with $0<\varepsilon\ll1$, converges to the following averaged equation:
\begin{align} \label{MASDE}
 d\bar{X}(t)=\bar{b}\left(\bar{X}(t),\mathscr{L}_{\bar{X}(t)}\right)dt+\bar{\sigma}\left(\bar{X}(t),\mathscr{L}_{\bar{X}(t)}\right)dW(t)+\int_{U}\bar{h}\left(\bar{X}(t),\mathscr{L}_{\bar{X}(t)},z\right)\tilde{N}(dt,dz),\;\;\bar{X}(0)=x_0,
\end{align}
in certain sense, by taking the assumption of proper averaging conditions. 
For more details of \eqref{MASDE}, see Section \ref{sec3}.

Again, we have to point out that, compared with the case of classical SDEs, there are much few results on the averaging principle for McKean-Vlasov SDEs due to their distribution-dependent feature. And the existing references on averaging principles for McKean-Vlasov SDEs mainly focus on Brownian cases \cite{Xu2021,Shen2022}. For some interesting results with other types of noise, e.g., fractional Brownian noise, we refer to \cite{Shen2022av}. Nevertheless, to the best of authors’ knowledge, the averaging principle for McKean-Vlasov SDEs with L\'evy noise has not yet been considered to date. This inspires us to establish an averaging principle.

The rest of this paper is arranged as follows. In Section \ref{sec2}, we devote to investigating the existence and uniqueness of solutions to a class of McKean-Vlasov SDEs with L\'evy-type perturbations. In Section \ref{sec3}, we prove an averaging principle for the solutions of the considered McKean-Vlasov SDEs. In Section \ref{sec4}, we present a specific example to illustrate the theoretical results of this article. And we postpone the details of the proof for Lemma \ref{lemma3-1} and the propagation of chaos result to the end as Appendix.


\renewcommand{\theequation}{\thesection.\arabic{equation}}
\setcounter{equation}{0}

\section{Well-posedness of L\'evy-type McKean-Vlasov stochastic differential equations}\label{sec2}

 We start with some notations used in the sequel. Let $|\cdot|$ and $\langle \cdot, \cdot\rangle$ be the Euclidean vector norm and the scalar product in $\mathbb{R}^d$, respectively. For a matrix $A$, we use the Frobenius norm by $\|A\|=\sqrt{tr[AA^T]}$, where $A^T$ stands for the transpose of the matrix $A$. Let $\mathcal{M}(\mathbb{R}^d)$ represent the space of all probability measures on $\mathbb{R}^d$ carrying the usual topology of weak convergence. Furthermore, for $p\geqslant 1$, let $\mathcal{M}_p(\mathbb{R}^d)$ denote the subspace of $\mathcal{M}(\mathbb{R}^d)$ as follows:
 $$\mathcal{M}_p(\mathbb{R}^d):=\left\{\mu\in\mathcal{M}(\mathbb{R}^d): \mu(|\cdot|^p):=\int_{\mathbb{R}^d}|x|^p\mu(dx)<\infty\right\}.$$
 For $\mu_1, \mu_2\in\mathcal{M}_p(\mathbb{R}^d)$, the $L^p$-Wasserstein metric between $\mu_1$ and $\mu_2$ is defined by
 $$W_p(\mu_1,\mu_2):=\inf_{\pi\in\mathscr{C}(\mu_1,\mu_2)}\left(\int_{\mathbb{R}^d\times \mathbb{R}^d}|x-y|^p\pi(dx,dy)\right)^{\frac{1}{p}},$$
where $\mathscr{C}(\mu_1,\mu_2)$ means the collection of all the probability measures whose marginal distributions are $\mu_1$, $\mu_2$, respectively. Then $\mathcal{M}_p(\mathbb{R}^d)$ endowed with above metric is a Polish space. 

Let $\delta_x$ be Dirac delta measure centered at the point $x\in\mathbb{R}^d$. A direct calculation shows the $\delta_x$ belongs to $\mathcal{M}_p(\mathbb{R}^d)$ for any $x\in\mathbb{R}^d$. Another remark is that, if $\mu_1=\mathscr{L}_X$, $\mu_2=\mathscr{L}_Y$ are the corresponding distributions of random variables $X$ and $Y$ respectively, then 
$$(W_p(\mu_1,\mu_2))^p\leqslant \int_{\mathbb{R}^d\times \mathbb{R}^d}|x-y|^p\mathscr{L}_{(X,Y)}(dx,dy)=\mathbb{E}|X-Y|^p,$$
in which $\mathscr{L}_{(X,Y)}$ represents the joint distribution of random vector $(X,Y)$. 

Given $T>0$, let $D([0,T];\mathbb{R}^d)$ be the collection of all c\`{a}dl\`{a}g (i.e., right continuous with left limits) functions from $[0,T]$ to $\mathbb{R}^d$ equipped with the supremum norm. For $p\geqslant1$, we use $L^p(\Omega;\mathbb{R}^d)$ to denote the family of all $\mathbb{R}^d$-valued random variables $X$ with $\mathbb{E}|X|^p<\infty$. Analogously, we denote by $L^p(\Omega;D([0,T];\mathbb{R}^d))$ the subspace of all $D([0,T];\mathbb{R}^d)$-valued random variables, which satisfy $\mathbb{E}[\sup_{0\leqslant t\leqslant T}|X(t)|^p]<\infty.$
Equipping with the following norm 
$$||X||_{L^p}:=\left(\mathbb{E}\left[\sup_{0\leqslant t\leqslant T}|X(t)|^p\right]\right)^{\frac{1}{p}},$$
the space $L^p(\Omega;D([0,T];\mathbb{R}^d))$ is a Banach space.

\subsection{Formulation of the well-posedness results}

Keep in mind that this section is dedicated to establishing the existence and uniqueness theorem for the solutions of $d$-dimensional L\'evy-type McKean-Vlasov SDEs, i.e.,
\begin{align}\label{Main-equation}
&dX(t)=b\left(t,X(t),\mathscr{L}_{X(t)}\right) dt+\sigma\left(t,X(t),\mathscr{L}_{X(t)}\right)dW(t)+\int_{U}h\left(t,X(t),\mathscr{L}_{X(t)},z\right)\tilde{N}(dt,dz)
\end{align}
on $t\in[0,T]$ with initial data $X(0)=x_0$, 
where $$b:[0,T]\times\mathbb{R}^d\times M(\mathbb{R}^d)\to\mathbb{R}^d, \;\sigma:[0,T]\times\mathbb{R}^d\times M(\mathbb{R}^d)\to\mathbb{R}^{d\times m},\; h:[0,T]\times\mathbb{R}^d\times M(\mathbb{R}^d)\times U\to\mathbb{R}^d,$$
are Borel measurable functions. 
Let us first give the precise definition of a solution to (\ref{Main-equation}).


\begin{definition}
We say that (\ref{Main-equation}) admits a unique strong solution if there exists an $\{\mathcal{F}_t\}_{0\leqslant t\leqslant T}$-adapted $\mathbb{R}^d$-valued c\`{a}dl\`{a}g stochastic process $\{X(t)\}_{0\leqslant t\leqslant T}$ such that
\begin{enumerate}
\item[(i)]
$X(t)=x_0+\int_0^t b\left(s,X(s),\mathscr{L}_{X(s)}\right) ds+\int_0^t\sigma\left(s,X(s),\mathscr{L}_{X(s)}\right)dW(s)+\int_0^t\int_{U}h\left(s,X(s),\mathscr{L}_{X(s)},z\right)\tilde{N}(ds,dz), \; t\in[0,T], \; \mathbb{P}-a.s.;$
\item[(ii)]
if $Y=\{Y(t)\}_{0\leqslant t\leqslant T}$ is another solution with $Y(0)=x_0$, then 
$$\mathbb{P}(X(t)=Y(t) \;\hbox{ for all} \;0\leqslant t\leqslant T)=1.$$
\end{enumerate}
\end{definition}


Assume that there exists $\kappa\geqslant2$ such that the following assumptions hold.

\noindent{\bf A1. }(One-sided local Lipschitz condition on the state variable) For every $R>0$, there exists a constant $L_R>0$ such that for any $t\in[0,T]$, $x,y\in\mathbb{R}^d$ with $|x|\vee|y|\leqslant R$, and $\mu\in\mathcal{M}_2(\mathbb{R}^d)$,
$$\langle x-y, b(t,x,\mu)-b(t,y,\mu)\rangle\vee\|\sigma(t,x,\mu)-\sigma(t,y,\mu)\|^2\vee\int_U\left|h(t,x,\mu,z)-h(t,y,\mu,z)\right|^2\nu(dz)\leqslant L_R|x-y|^2.$$
{\bf A2. }(Global Lipschitz condition on the measure) There exists $L>0$ such that, for any $t\in[0,T]$, $x\in\mathbb{R}^d$ and $\mu_1,\mu_2\in\mathcal{M}_2(\mathbb{R}^d)$,
$$|b(t,x,\mu_1)-b(t,x,\mu_2)|^2+\|\sigma(t,x,\mu_1)-\sigma(t,x,\mu_2)\|^2+\int_U\left|h(t,x,\mu,z)-h(t,y,\mu,z)\right|^2\nu(dz)\leqslant LW_2^2(\mu_1,\mu_2).$$ 
\noindent{\bf A3. }(Continuity) For any $t\in[0,T]$, $b(t,\cdot,\cdot), \sigma(t,\cdot,\cdot), \int_Uh(t,\cdot,\cdot,z)\nu(dz)$ is continuous on $\mathbb{R}^d\times\mathcal{M}_2(\mathbb{R}^d).$\par
\noindent{\bf A4. }(One-sided linear \& global linear growth condition) There exists $K>0$ such that, for any $t\in[0,T]$, $x\in\mathbb{R}^d$ and $\mu\in\mathcal{M}_2(\mathbb{R}^d)$,
$$\langle x, b(t,x,\mu)\rangle\vee \|\sigma(t,x,\mu)\|^2\vee \int_U|h(t,x,\mu,z)|^2\nu(dz)\leqslant K(1+|x|^2+W_2^2(\mu,\delta_0)).$$ 
\noindent{\bf A5. }($\kappa$-order growth condition w.r.t. the drift coefficient $b$) There exists positive $K_1$ such that, for any $t\in[0,T]$, $x\in\mathbb{R}^d$ and $\mu\in\mathcal{M}_{2}(\mathbb{R}^d)$,
$$|b(t,x,\mu)|^2\leqslant K_1(1+|x|^{\kappa}+W_2^{\kappa}(\mu,\delta_0)).$$ 

\noindent{\bf A6. }($r$-order moment condition for the initial data) Consider $x_0\in L^r(\Omega;\mathbb{R}^d)$ for some  $r\geqslant \max\{\kappa^2/2,4\}$, i.e., $\mathbb{E}|x_0|^r<\infty.$

\noindent{\bf A7. }(Additional growth conditions w.r.t. the jump coefficient $h$) There exists a positive $K_2$ such that, for any $t\in[0,T]$, $x\in\mathbb{R}^d$ and $\mu\in\mathcal{M}_{2}(\mathbb{R}^d)$,
$$
\int_U|h(t,x,\mu,z)|^r\nu(dz)\leqslant K_2(1+|x|^r+W_2^{r}(\mu,\delta_0)).
$$
In addition, if $\kappa >2$, there exist constants $K_3, L^{'}>0$, such that, for any $t\in[0,T]$, $x,y\in\mathbb{R}^d$ and $\mu,\mu_1,\mu_2\in\mathcal{M}_{2}(\mathbb{R}^d)$,

$$
\int_U|h(t,x,\mu,z)|^\kappa\nu(dz)\leqslant K_3(1+|x|^\kappa+W_2^{\kappa}(\mu,\delta_0)),
$$
$$\int_U\left|h(t,x,\mu_1,z)-h(t,y,\mu_2,z)\right|^{\kappa}\nu(dz)\leq L^{'}W_2^{\kappa}(\mu_1,\mu_2),$$
and for every $R>0$, there exists a constant $L_R^{'}>0$ such that for any $t\in[0,T]$, $x,y\in\mathbb{R}^d$ with $|x|\vee|y|\leqslant R$, and $\mu\in\mathcal{M}_2(\mathbb{R}^d)$,
$$
\int_U\left|h(t,x,\mu,z)-h(t,y,\mu,z)\right|^{\kappa}\nu(dz)\leqslant L_R^{'}|x-y|^{\kappa}.
$$

The main result of this section is stated as follows.

\begin{theorem}\label{mainresult1} {\bf (Well-posedness)}
Let Assumptions {\bf A1}-{\bf A7} be satisfied. Then 
equation (\ref{Main-equation}) admits a unique strong solution $X(t)\in L^{\kappa}(\Omega;\mathbb{R}^d)$ for $t\in[0,T]$ with the initial value $X(0)=x_0$. Moreover, the following estimate holds:
\begin{equation}
\mathbb{E}\left[\sup_{0\leqslant t\leqslant T}|X(t)|^r\right]\leqslant C,
\end{equation}
with $C:=C(T,r,\mathbb{E}|x_0|^r)$ is a positive number. Here $\kappa\geqslant2$ and $r\geqslant \max\{\kappa^2/2,4\}$.
\end{theorem}

\begin{remark}
We point out that the conditions in Assumptions {\bf A1}-{\bf A7} are carefully chosen and the results in Theorem \ref{mainresult1} are generally applicable:
\begin{itemize}
    \item[(i)] The one-sided local Lipschitz condition in {\bf A1} is weaker than the classical local Lipschitz condition. In fact, it is clear that the local Lipschitz condition implies the one-sided local Lipschitz one (by applying the mean value inequality directly). But it is not vice versa. For example, let $b(t,x,\mu)=x^3-x^{\frac{1}{3}}+t+\int_{\mathbb{R}} z\mu(dz)$ in $\mathbb{R}$. Note that, for $|x|\vee|y|\leqslant R$,
    $$
    \langle x-y, b(t,x,\mu)-b(t,y,\mu)\rangle=|x-y|^2(x^2+xy+y^2)-(x-y)(x^{\frac{1}{3}}-y^{\frac{1}{3}})\leqslant 3R^2|x-y|^2,
    $$
    as $(x-y)(x^{\frac{1}{3}}-y^{\frac{1}{3}})\geqslant0$ for all $x,y$. We can find that $b$ is one-sided locally Lipschitz but not locally Lipschitz.
    \item[(ii)] Comparing with the one-sided (global) Lipschitz condition in the recently paper \cite{Neelima2020}, i.e., there exists $C>0$ such that for any $x,y\in\mathbb{R}^d$, $\mu\in\mathcal{M}_2(\mathbb{R}^d),$
    $$
    \langle x-y, b(t,x,\mu)-b(t,y,\mu)\rangle+\|\sigma(t,x,\mu)-\sigma(t,y,\mu)\|^2+\int_U\left|h(t,x,\mu,z)-h(t,y,\mu,z)\right|^2\nu(dz)\leqslant C|x-y|^2,
    $$
    the one-sided local Lipschitz condition in {\bf A1} is expressed via the operations ``$\vee$" instead of ``$+$". Even so, such a ``local" condition may include some weaker cases. For example, let $b$ be a one-sided locally Lipschitz function and $\sigma=h=x$ with $\nu(U)<\infty$. Then {\bf A1} holds but the the one-sided (global) Lipschitz condition in \cite{Neelima2020} is failed.

    \item[(iii)] The result is degenerated into a pure-Brownian one if we set $h\equiv 0$. Different from the previous Brownian case paper \cite{Li2022strong} in which the drift coefficient need to satisfy the linear growth condition, the drift coefficient $b$ here only satisfies the one-sided linear growth condition, and may be polynomial growth w.r.t. the state variable (see {\bf A4-A5}).
\end{itemize}
\end{remark}

\subsection{Euler type approximation and auxiliary lemmas}

One key of our method to prove Theorem \ref{mainresult1} is constructing an Euler-like sequence for the McKean-Vlasov SDEs \eqref{Main-equation}. Once we show that this sequence is Cauchy in an appropriate  complete space (i.e., $L^{\kappa}(\Omega;D([0,T];\mathbb{R}^d))$, as we will see later), we are able to conclude that there exists a limiting process which is indeed the desired solution to \eqref{Main-equation}.

To this end, given $T>0$, we consider the equidistant partitions of $[0,T]$. For any integer $n\geqslant1$, we set $\triangle_n=\frac{T}{n}$ and $t_k^n=k\triangle_n$, $k=0,1,\ldots, n$. In this way, for fixed $k$ ($0\leqslant k \leqslant n-1$) and $t\in(t_k^n, t_{k+1}^n]$, we consider 
\begin{equation}\label{Approximation-k}
dX^{(n)}(t)=b\left(t,X^{(n)}(t),\mu^{(n)}_{t_k^n}\right)dt+\sigma\left(t,X^{(n)}(t),\mu^{(n)}_{t_k^n}\right)dW(t)+\int_{U}h\left(t,X^{(n)}(t),\mu^{(n)}_{t_k^n},z\right)\tilde{N}(dt,dz),
\end{equation}
where $\mu^{(n)}_{t_k^n}=\mathscr{L}_{X^{(n)}(t_k^n)}$. It is clear that, for each fixed $k$, if the initial value $X^{(n)}(t_k^n)$ and the distribution $\mu^{(n)}_{t_k^n}$ (at the left point $t_k^n$) are known, then \eqref{Approximation-k} is a standard SDE independent of the law of $X^{(n)}(t)$. We now prove inductively the existence and uniqueness of the solution to \eqref{Approximation-k}.

In fact, for $k=0$ and $t\in[0,t_1^n]$, the distribution is $ \mu^{(n)}_{0}=\mathscr{L}_{X^{(n)}(0)}=\mathscr{L}_{x_0}$. Applying Assumptions {\bf A1}$\&${\bf A4}, we observe that the coefficients in \eqref{Approximation-k} (with $k=0$) satisfy 
$$\left\langle x-y, b\left(t,x,\mu^{(n)}_{0}\right)-b\left(t,y,\mu^{(n)}_{0}\right)\right\rangle+\left\|\sigma\left(t,x,\mu^{(n)}_{0}\right)-\sigma\left(t,y,\mu^{(n)}_{0}\right)\right\|^2+\int_U\left|h\left(t,x,\mu^{(n)}_{0},z\right)-h\left(t,y,\mu^{(n)}_{0},z\right)\right|^2\nu(dz)\leqslant 3L_R|x-y|^2$$
and 
\begin{equation}
\begin{split}
\left\langle x, b\left(t,x,\mu^{(n)}_{0}\right)\right\rangle+\left\|\sigma\left(t,x,\mu^{(n)}_{0}\right)\right\|^2+\int_U\left|h\left(t,x,\mu^{(n)}_{0},z\right)\right|^2\nu(dz)&\leqslant 3K\left(1+|x|^2+W_2^2\left(\mu^{(n)}_{0},\delta_0\right)\right)\\
&\leqslant3K\left(1+|x|^2\right)\left(1+\mathbb{E}\left|X^{(n)}(0)\right|^2\right).\notag
\end{split}
\end{equation}
Referring to Theorem 1.1 in \cite{Majka2016note}, it admits a unique solution on $[0,t_1^n]$. Furthermore, by Assumption {\bf A5}, it follows that for $r\geqslant\max\{\frac{\kappa^2}{2},4\}$, there exits a positive constant $C$ such that 
\begin{equation}
\mathbb{E}\left[\sup_{0\leqslant t\leqslant t_1^n}\left|X^{(n)}(t)\right |^r \right]\leqslant C\left(1+\mathbb{E}\left|X^{(n)}(0)\right|^r\right), \notag
\end{equation}
whose proof is quite similar to Lemma \ref{UBP} below, and we omit the details here. 
Therefore, we can define $X^{(n)}(t_1^n)$ (which satisfies $\mathbb{E}|X^{(n)}(t_1^n)|^r<\infty$) and $\mu^{(n)}_{t_1^n}=\mathscr{L}_{X^{(n)}(t_1^n)}$.

For $k=1$ and $t\in(t_1^n,t_2^n]$, we can use $(X^{(n)}(t_1^n),\mu^{(n)}_{t_1^n})$ in place of $(X^{(n)}(0),\mu^{(n)}_{0})$ and repeat the above procedure. Inductively, for any  $k=0,1,\ldots, n-1$ and $t\in(t_k^n, t_{k+1}^n]$, we obtain the existence and uniqueness of the solution to the SDE \eqref{Approximation-k} as well as the corresponding estimate 
\begin{equation}\label{estimate-k}
\mathbb{E}\left[\sup_{t_k^n\leqslant t\leqslant t_{k+1}^n}\left|X^{(n)}(t)\right|^r\right]\leqslant C\left(1+\mathbb{E}\left|X^{(n)}(t_k^n)\right|^r\right),
\end{equation}
by similar arguments. 

At this point, we define by $[t]_n=t_k^n$ for all $t\in (t_k^n,t_{k+1}^n]$, $k=0,1,\ldots, n-1$. Then, for $t\in[0,T]$,  we introduce the following approximating SDE 
\begin{equation}\label{Approximation-Eq}
dX^{(n)}(t)=b\left(t,X^{(n)}(t),\mu^{(n)}_{[t]_{n}}\right)dt+\sigma\left(t,X^{(n)}(t),\mu^{(n)}_{[t]_{n}}\right)dW(t)+\int_{U}h\left(t,X^{(n)}(t),\mu^{(n)}_{[t]_{n}},z\right)\tilde{N}(dt,dz),
\end{equation}
with the initial value $X^{(n)}(0)=x_0$, where $\mu^{(n)}_{[t]_{n}}=\mathscr{L}_{X^{(n)}([t]_{n})}$. According to the above procedures and results for \eqref{Approximation-k}, we conclude that there exists a unique solution to \eqref{Approximation-Eq}. In fact, for each $n\geqslant 1$ and $t\in[0,T]$, we can always find certain $k_\ast$ ($0\leqslant k_\ast \leqslant n-1$) such that $t\in (t_{k_\ast}^n,t_{{k_\ast}+1}^n]$. Then, the solution to \eqref{Approximation-Eq} can be written as
\begin{align}
X^{n}(t)=&x_0+\sum_{k=0}^{k_\ast}\int_{t_k^n}^{t_{k+1}^n \wedge t}
b\left(t,X^{(n)}(s),\mu^{(n)}_{t_k^n}\right)ds+\sigma\left(s,X^{(n)}(s),\mu^{(n)}_{t_k^n}\right)dW(s)+\int_{U}h\left(s,X^{(n)}(s),\mu^{(n)}_{t_k^n},z\right)\tilde{N}(ds,dz) \notag\\
=&X^{n}(t_1^n)+\sum_{k=1}^{k_\ast}\int_{t_k^n}^{t_{k+1}^n \wedge t}
b\left(t,X^{(n)}(s),\mu^{(n)}_{t_k^n}\right)ds+\sigma\left(s,X^{(n)}(s),\mu^{(n)}_{t_k^n}\right)dW(s)+\int_{U}h\left(s,X^{(n)}(s),\mu^{(n)}_{t_k^n},z\right)\tilde{N}(ds,dz) \notag\\
&\cdots \;\;\;\; \cdots\notag\\
=&X^{n}(t_{k_\ast}^n)+\int_{t_{k_\ast}^n}^{t}
b\left(t,X^{(n)}(s),\mu^{(n)}_{t_{k_\ast}^n}\right)ds+\sigma\left(s,X^{(n)}(s),\mu^{(n)}_{t_{k_\ast}^n}\right)dW(s)+\int_{U}h\left(s,X^{(n)}(s),\mu^{(n)}_{t_{k_\ast}^n},z\right)\tilde{N}(ds,dz),\notag
\end{align}
and it is well defined based on the results for \eqref{Approximation-k} with $k=0,1,\ldots,k_\ast$.
Moreover, we have the following estimate 
\begin{equation}
\mathbb{E}\left[\sup_{0\leqslant t\leqslant T}\left|X^{(n)}(t)\right|^r\right]=\mathbb{E}\left[\max_{0\leqslant k\leqslant n-1}\sup_{t_k^n\leqslant t\leqslant t_{k+1}^{n}}\left|X^{(n)}(t)\right|^r\right]\leqslant \sum_{k=0}^{n-1}\mathbb{E}\left[\sup_{t_k^n\leqslant t\leqslant t_{k+1}^{n}}\left|X^{(n)}(t)\right|^r\right]\leqslant C(n)<\infty.
\end{equation}
Notice that, under Assumption {\bf A6}, i.e., the initial data $x_0$ satisfies $\mathbb{E}|x_0|^r<\infty$ with $r\geqslant\max\{\frac{\kappa^2}{2},4\}\geqslant\kappa$, we conclude that $X^{(n)}\in L^{r}(\Omega;D([0,T];\mathbb{R}^d))\subset L^{\kappa}(\Omega;D([0,T];\mathbb{R}^d))$. That is, the stochastic processes  $\{X^{(n)}(t)\}_{n\geqslant1}$ given by \eqref{Approximation-Eq} form a sequence in $L^{\kappa}(\Omega;D([0,T];\mathbb{R}^d))$. Next, to show that this sequence is indeed Cauchy, we need the following two useful lemmas.

\begin{lemma}\label{UBP} (Uniform boundedness property)
Under Assumptions {\bf A4, A6} and {\bf A7}, for any $T>0$, there is a positive constant $C_r$ (which is independent of $n$) such that
\begin{equation}\label{UBeq}
\mathbb{E}\left[\sup_{0\leqslant t\leqslant T}\left|X^{(n)}(t)\right|^r\right]\leqslant C_r.
\end{equation}
\end{lemma}

\begin{proof}
For $r\geqslant\max\{\frac{\kappa^2}{2},4\}$ and $t\in[0,T]$, applying It\^o's formula to $|x|^r$ yields that,
\begin{align}\label{ItoEq}
&\left|X^{(n)}(t)\right|^r\notag\\
=&|x_0|^r+r\int_0^{t}\left|X^{(n)}(s)\right|^{r-2}\left\langle X^{(n)}(s), b\left(s,X^{(n)}(s),\mu^{(n)}_{[s]_n}\right)\right\rangle ds+\frac{r}{2}\int_0^{t}\left|X^{(n)}(s)\right|^{r-2} \left\|\sigma(s,X^{(n)}(s),\mu^{(n)}_{[s]_n})\right\|^2 ds\notag\\
&+\frac{r(r-2)}{2}\int_0^{t}\left|X^{(n)}(s)\right|^{r-4} \left\|(X^{(n)}(s))^T\sigma\left(s,X^{(n)}(s),\mu^{(n)}_{[s]_n}\right)\right\|^2 ds\notag\\
&+r\int_0^{t}\left|X^{(n)}(s)\right|^{r-2}\left\langle X^{(n)}(s), \sigma\left(s,X^{(n)}(s),\mu^{(n)}_{[s]_n}\right)dW(s)\right\rangle\notag\\
&+r\int_0^{t}\int_U\left|X^{(n)}(s)\right|^{r-2}\left\langle X^{(n)}(s), h\left(s,X^{(n)}(s),\mu^{(n)}_{[s]_n},z\right)\right\rangle \tilde{N}(ds,dz)\notag\\
&+\int_0^{t}\int_U\left[\left|X^{(n)}(s)+h\left(s,X^{(n)}(s),\mu^{(n)}_{[s]_n},z\right)\right|^{r}-\left|X^{(n)}(s)\right|^r-r\left|X^{(n)}(s)\right|^{r-2}\left\langle X^{(n)}(s), h\left(s,X^{(n)}(s),\mu^{(n)}_{[s]_n},z\right)\right\rangle\right] N(ds,dz).
\end{align}
By virtue of Assumption {\bf A4} and Young's inequality 
\begin{equation}\label{YE}
ab\leqslant\epsilon \frac{a^p}{p}+\epsilon^{-\frac{q}{p}}\frac{b^q}{q}\; \;\; \hbox{for all}\; \epsilon, a, b>0,\;\hbox{where}\; p>1,\; \frac{1}{p}+\frac{1}{q}=1, 
\end{equation}
one can estimate the second term of (\ref{ItoEq}) by
\begin{align}\label{term2}
&rK\int_0^{t}\left|X^{(n)}(s)\right|^{r-2}\left(1+\left|X^{(n)}(s)\right|^2+\mathbb{E}\left|X^{(n)}([s]_{n})\right|^2\right)ds\notag\\
\leqslant& (r-2)K\int_0^t\left|X^{(n)}(s)\right|^rds+2K\int_0^t\left(1+\left|X^{(n)}(s)\right|^2+\mathbb{E}\left|X^{(n)}([s]_{n})\right|^2\right)^{\frac{r}{2}}ds\notag\\
\leqslant& (r-2)K\int_0^t\left|X^{(n)}(s)\right|^rds+2\cdot3^{\frac{r}{2}-1}K\int_0^t\left(1+\left|X^{(n)}(s)\right|^r+\mathbb{E}\left|X^{(n)}([s]_{n})\right|^r\right)ds.
\end{align}
Here, we take $\epsilon=1$, $p=\frac{r}{r-2}$, $q=\frac{r}{2}$, and the last step is based on the elementary inequality $(a+b+c)^{l}\leqslant 3^{l-1}(|a|^l+|b|^l+|c|^l)$ for all $a,b,c\in\mathbb{R}$, $l\geqslant1$ and H\"older inequality. Analogously, the third and fourth terms of (\ref{ItoEq}) can be estimated by
\begin{equation}\label{term3}
\frac{r-2}{2}K\int_0^t\left|X^{(n)}(s)\right|^rds+3^{\frac{r}{2}-1}K\int_0^t\left(1+\left|X^{(n)}(s)\right|^r+\mathbb{E}\left|X^{(n)}([s]_{n})\right|^r\right)ds,
\end{equation}
and
\begin{equation}\label{term4}
\frac{(r-2)^2}{2}K\int_0^t\left|X^{(n)}(s)\right|^rds+3^{\frac{r}{2}-1}(r-2)K\int_0^t\left(1+\left|X^{(n)}(s)\right|^r+\mathbb{E}\left|X^{(n)}([s]_{n})\right|^r\right)ds,
\end{equation}
respectively.
Further, note that the map $y\to|y|^r$ is of class $C^2$, by the remainder formula, for any $y, b\in\mathbb{R}^d$, one obtains
\begin{equation}\label{RF}
|y|^r-|b|^r-r|b|^{r-2}\langle b, y-b\rangle\leqslant C_1\int_0^1|y-b|^2|b+\theta(y-b)|^{r-2}d\theta\leqslant C_1(|b|^{r-2}|y-b|^2+|y-b|^r),
\end{equation}
and thus the last term on the right hand side of (\ref{ItoEq}) can be estimated as
\begin{equation}\label{termN}
C_1\int_0^{t}\int_U\left(\left|X^{(n)}(s)\right|^{r-2}\left|h\left(s,X^{(n)}(s),\mu^{(n)}_{[s]_n},z\right)\right|^2+\left|h\left(s,X^{(n)}(s),\mu^{(n)}_{[s]_n},z\right)\right|^{r}\right)N(ds,dz).
\end{equation}
Denote the above estimate (\ref{termN}) by $N_t$. Substituting (\ref{term2})-(\ref{termN}) into (\ref{ItoEq}), taking suprema over $[0,u]$ for $u\in[0,T]$ and then taking expectations gives that
\begin{align}\label{Te}
\mathbb{E}\sup_{0\leqslant t\leqslant u}|X^{(n)}(t)|^r\leqslant&\mathbb{E}|x_0|^r+\mathbb{E}\sup_{0\leqslant t\leqslant u}|M_t|+\mathbb{E}\sup_{0\leqslant t\leqslant u}|N_t|+\frac{(r+1)(r-2)}{2}K\int_0^u\mathbb{E}\left|X^{(n)}(s)\right|^rds\notag\\
&+3^{\frac{r}{2}-1}(r+1)K\int_0^u\left(1+\mathbb{E}\left|X^{(n)}(s)\right|^r+\mathbb{E}\left|X^{(n)}([s]_{n})\right|^r\right)ds,
\end{align}
where  
$$
M_t:=r\int_0^{t}\left|X^{(n)}(s)\right|^{r-2}\left\langle X^{(n)}(s), \sigma\left(s,X^{(n)}(s),\mu^{(n)}_{[s]_n}\right)dW(s)\right\rangle+r\int_0^{t}\int_U\left|X^{(n)}(s)\right|^{r-2}\left\langle X^{(n)}(s), h\left(s,X^{(n)}(s),\mu^{(n)}_{[s]_n},z\right)\right\rangle \tilde{N}(ds,dz)
$$
is indeed a local martingale. 
On the one hand, by the Burkholder-Davis-Gundy inequality (see, e.g., Theorem 7.3 of Chapter 1 in \cite{Mao2007} or Lemma 2.1 of \cite{Dareiotis2016}), there exists a constant $C_2>0$ such that
\begin{align}
\mathbb{E}\sup_{0\leqslant t\leqslant u}|M_t|&\leqslant C_2r\mathbb{E}\left[\int_0^{u}\left|X^{(n)}(s)\right|^{2r-2}\left\|\sigma\left(s,X^{(n)}(s),\mu^{(n)}_{[s]_n}\right)\right\|^2ds+\int_0^{u}\int_U\left|X^{(n)}(s)\right|^{2r-2}\left|h\left(s,X^{(n)}(s),\mu^{(n)}_{[s]_n}\right)\right|^2\nu(dz)ds\right]^{\frac{1}{2}}\notag\\
&\leqslant C_2r\mathbb{E}\left[\sup_{0\leqslant t\leqslant u}\left|X^{(n)}(t)\right|^{r-1}\left(\int_0^{u}\left\|\sigma\left(s,X^{(n)}(s),\mu^{(n)}_{[s]_n}\right)\right\|^2ds+\int_0^{u}\int_U\left|h\left(s,X^{(n)}(s),\mu^{(n)}_{[s]_n}\right)\right|^2\nu(dz)ds\right)^{\frac{1}{2}}\right],\notag
\end{align}
which on the application of Assumption {\bf A4} gives
\begin{equation*}
\mathbb{E}\sup_{0\leqslant t\leqslant u}|M_t|\leqslant C_2r\mathbb{E}\left[\sup_{0\leqslant t\leqslant u}\left|X^{(n)}(t)\right|^{r-1}\left(2K\int_0^u\left(1+\left|X^{(n)}(s)\right|^2+\mathbb{E}\left|X^{(n)}([s]_n)\right|^2\right)ds\right)^{\frac{1}{2}}\right].
\end{equation*}
Then, due to Young's inequality given in (\ref{YE}) (letting $\epsilon=\frac{1}{2C_2(r-1)}$, $p=\frac{r}{r-1}$, $q=r$), H\"older inequality,  elementary inequality and Lyapunov inequality, one further has
\begin{align}\label{Me}
\mathbb{E}\sup_{0\leqslant t\leqslant u}|M_t|
&\leqslant \frac{1}{2}\mathbb{E}\sup_{0\leqslant t\leqslant u}\left|X^{(n)}(t)\right|^r+C_2^r(2(r-1))^{r-1}(2K)^{\frac{r}{2}}\mathbb{E}\left[\int_0^u\left(1+\left|X^{(n)}(s)\right|^2+\mathbb{E}\left|X^{(n)}([s]_n)\right|^2\right)ds\right]^{\frac{r}{2}}\notag\\
&\leqslant \frac{1}{2}\mathbb{E}\sup_{0\leqslant t\leqslant u}\left|X^{(n)}(t)\right|^r+C_2^r(2(r-1))^{r-1}(2K)^{\frac{r}{2}}u^{\frac{r}{2}-1}\mathbb{E}\left[\int_0^u\left(1+\left|X^{(n)}(s)\right|^2+\mathbb{E}\left|X^{(n)}([s]_n)\right|^2\right)^{\frac{r}{2}}ds\right]\notag\\
&\leqslant \frac{1}{2}\mathbb{E}\sup_{0\leqslant t\leqslant u}\left|X^{(n)}(t)\right|^r+C_2^r(2(r-1))^{r-1}(2K)^{\frac{r}{2}}(3u)^{\frac{r}{2}-1}\mathbb{E}\int_0^u\left(1+\left|X^{(n)}(s)\right|^r+\mathbb{E}\left|X^{(n)}([s]_n)\right|^r\right)ds.
\end{align}
On the other hand, by Assumptions {\bf A4}$\&${\bf A7}, Young's inequality (\ref{YE}) (with $\epsilon=1$, $p=\frac{r}{r-2}$, $q=\frac{r}{2}$), elementary inequality and Lyapunov inequality,
\begin{align}\label{Ne}
\mathbb{E}\sup_{0\leqslant t\leqslant u}|N_t|
&\leqslant C_1\mathbb{E}\int_0^{u}\int_U\left(\left|X^{(n)}(s)\right|^{r-2}\left|h\left(s,X^{(n)}(s),\mu^{(n)}_{[s]_n},z\right)\right|^2+\left|h\left(s,X^{(n)}(s),\mu^{(n)}_{[s]_n},z\right)\right|^{r}\right)\nu(dz)ds\notag\\
&\leqslant C_1\mathbb{E}\int_0^{u}\left[K\left|X^{(n)}(s)\right|^{r-2}\left(1+\left|X^{(n)}(s)\right|^2+\mathbb{E}\left|X^{(n)}([s]_{n})\right|^2\right)+K_2\left(1+\left|X^{(n)}(s)\right|^{r}+\left(\mathbb{E}\left|X^{(n)}([s]_{n})\right|^2\right)^{\frac{r}{2}}\right)\right]ds\notag\\
&\leqslant C_1K\frac{r-2}{r}\int_0^u\mathbb{E}\left|X^{(n)}(s)\right|^rds+C_1\cdot\left(\frac{2K}{r}3^{\frac{r}{2}-1}+K_2\right)\int_0^u\left(1+\mathbb{E}\left|X^{(n)}(s)\right|^r+\mathbb{E}\left|X^{(n)}([s]_{n})\right|^r\right)ds.
\end{align}
As a result, combining all the above estimates (\ref{Me})-(\ref{Ne}) and applying the Gr\"onwall inequality, we get from (\ref{Te}) that
\begin{equation*}
\mathbb{E}\sup_{0\leqslant t\leqslant T}\left|X^{(n)}(t)\right|^r\leqslant2(1+\mathbb{E}|x_0|^r)e^{\left[K(r-2)\frac{r^2+r+2C_1}{r}+4\cdot3^{\frac{r}{2}-1}(r+1)+C_2^{r}2^{r+1}(r-1)^{r-1}(2K)^{\frac{r}{2}}(3T)^{\frac{r}{2}-1}+4C_1\cdot\left(\frac{2K}{r}3^{\frac{r}{2}-1}+K_2\right)\right]T}\leqslant C_r.
\end{equation*}
It is clear that the positive constant $C_r$ is dependent on $r, T, K, K_2$ and initial condition $x_0$, but independent of $n$. The proof is therefore complete.
\end{proof}

\begin{lemma}\label{HolderC} (Time H\"older continuity)
Let Assumptions {\bf A4-A7} hold. For any $x_0\in L^r(\Omega;\mathbb{R}^d)$ with $r\geqslant\kappa^2/2$, there exists a positive constant $C_{\kappa}$ such that, for any $0\leqslant s\leqslant t\leqslant T$ with $|t-s|\leqslant1$,
\begin{equation}
\sup_{n\geqslant1}\mathbb{E}\left[\left|X^{(n)}(t)-X^{(n)}(s)\right|^{\kappa}\right]\leqslant C_{\kappa}|t-s|.
\end{equation}
\end{lemma}

\begin{proof}
It follows from (\ref{Approximation-Eq}) that
$$X^{(n)}(t)-X^{(n)}(s)=\int_s^tb\left(u,X^{(n)}(u),\mu^{(n)}_{[u]_{n}}\right)du+\int_s^t\sigma\left(u,X^{(n)}(u),\mu^{(n)}_{[u]_{n}}\right)dW(u)+\int_s^t\int_{U}h\left(u,X^{(n)}(u),\mu^{(n)}_{[u]_{n}},z\right)\tilde{N}(du,dz).$$
By taking expectations on both sides, one gets
\begin{align}\label{timeDifference}
\mathbb{E}\left|X^{(n)}(t)-X^{(n)}(s)\right|^{\kappa}\leqslant&3^{\kappa-1}\mathbb{E}\left[\left|\int_s^tb\left(u,X^{(n)}(u),\mu^{(n)}_{[u]_{n}}\right)du\right|^{\kappa}\right]+3^{\kappa-1}\mathbb{E}\left[\left|\int_s^t\sigma\left(u,X^{(n)}(u),\mu^{(n)}_{[u]_{n}}\right)dW(u)\right|^{\kappa}\right]\notag\\
&+3^{\kappa-1}\mathbb{E}\left[\left|\int_s^t\int_{U}h\left(u,X^{(n)}(u),\mu^{(n)}_{[u]_{n}},z\right)\tilde{N}(du,dz)\right|^{\kappa}\right]:=B_1+B_2+B_3.
\end{align}
Subsequently we estimate $B_1$, $B_2$ and $B_3$ separately, and for readability, we only present the core estimating procedures here. By H\"older inequality, Assumption {\bf A5}, Lyapunov inequality and estimate (\ref{UBeq}), we have 
\begin{align}
B_1\leqslant& 3^{\kappa-1}(t-s)^{\kappa-1}\int_s^t\mathbb{E}\left[\left|b\left(u,X^{(n)}(u),\mu^{(n)}_{[u]_{n}}\right)\right|^{\kappa}\right]du\leqslant3^{\frac{3\kappa}{2}-2}K_1^{\frac{\kappa}{2}}(t-s)^{\kappa-1}\int_s^t\left(1+\mathbb{E}\left|X^{(n)}(u)\right|^{\frac{\kappa^2}{2}}+\left(\mathbb{E}\left|X^{(n)}([u]_n)\right|^2\right)^{\frac{\kappa^2}{4}}\right)du\notag\\
\leqslant&2\cdot3^{\frac{3\kappa}{2}-2}K_1^{\frac{\kappa}{2}}(t-s)^{\kappa-1}\int_s^t\left(1+\mathbb{E}\sup_{0\leqslant u\leqslant T}\left|X^{(n)}(u)\right|^{\frac{\kappa^2}{2}}\right)du\leqslant C_{\kappa}(t-s)^{\kappa}.\notag
\end{align}
By Burkholder-Davis-Gundy inequality, H\"older inequality, Assumption {\bf A4} and estimate (\ref{UBeq}), we get 
\begin{align}
B_2\leqslant& 3^{\kappa-1}M_{\kappa}\mathbb{E}\left[\left|\int_s^t\left\|\sigma\left(u,X^{(n)}(u),\mu^{(n)}_{[u]_{n}}\right)\right\|^2du\right|^{\frac{\kappa}{2}}\right]\leqslant3^{\kappa-1}M_{\kappa}(t-s)^{\frac{\kappa-2}{2}}\int_s^t\mathbb{E}\left[\left\|\sigma\left(u,X^{(n)}(u),\mu^{(n)}_{[u]_{n}}\right)\right\|^{\kappa}\right]du\notag\\
\leqslant&2\cdot3^{\frac{3\kappa}{2}-2}M_{\kappa}K^{\frac{\kappa}{2}}(t-s)^{\frac{\kappa-2}{2}}\int_s^t\left(1+\mathbb{E}\sup_{0\leqslant u\leqslant T}\left|X^{(n)}(u)\right|^{\kappa}\right)du\leqslant C_{\kappa}(t-s)^{\frac{\kappa}{2}},\notag
\end{align}
where $M_{\kappa}=[\kappa^{\kappa+1}/2(\kappa-1)^{\kappa-1}]^{\frac{\kappa}{2}}$.
Applying the Kunita's first inequality (see, e.g., Theorem 4.4.23 of \cite{Applebaum2009} or Lemma 2.1 of \cite{Dareiotis2016}), H\"older inequality, Assumptions {\bf A4$\&$A7} and estimate (\ref{UBeq}), we obtain
\begin{align}
B_3\leqslant& 3^{\kappa-1}D\mathbb{E}\left[\left|\int_s^t\int_U\left|h\left(u,X^{(n)}(u),\mu^{(n)}_{[u]_{n}},z\right)\right|^2\nu(dz)du\right|^{\frac{\kappa}{2}}\right]
+ 3^{\kappa-1}D\int_s^t\mathbb{E}\left[\int_U\left|h\left(u,X^{(n)}(u),\mu^{(n)}_{[u]_{n}},z\right)\right|^{\kappa}\nu(dz)\right]du\notag\\
\leqslant& 3^{\kappa-1}D(t-s)^{\frac{\kappa-2}{2}}\int_s^t\mathbb{E}\left[\left(\int_U\left|h\left(u,X^{(n)}(u),\mu^{(n)}_{[u]_{n}},z\right)\right|^2\nu(dz)\right)^{\frac{\kappa}{2}}\right]du
+ 3^{\kappa-1}D\int_s^t\mathbb{E}\left[\int_U\left|h\left(u,X^{(n)}(u),\mu^{(n)}_{[u]_{n}},z\right)\right|^{\kappa}\nu(dz)\right]du\notag\\
\leqslant& 3^{\kappa-1}DK^{\frac{\kappa}{2}}(t-s)^{\frac{\kappa-2}{2}}\int_s^t\mathbb{E}\left[\left(1+\left|X^{(n)}(u)\right|^2+\mathbb{E}X^{(n)}([u]_n)^2\right)^{\frac{\kappa}{2}}\right]du
+ 3^{\kappa-1}DK_3\int_s^t\mathbb{E}\left[\left(1+\left|X^{(n)}(u)\right|^{\kappa}+\left(\mathbb{E}X^{(n)}([u]_n)^2\right)^{\frac{\kappa}{2}}\right)\right]du\notag\\
\leqslant&2\cdot3^{\frac{3\kappa}{2}-2}DK^{\frac{\kappa}{2}}(t-s)^{\frac{\kappa-2}{2}}\int_s^t\left(1+\mathbb{E}\sup_{0\leqslant u\leqslant T}\left|X^{(n)}(u)\right|^{\kappa}\right)du+2\cdot3^{\kappa-1}DK_3\int_s^t\left(1+\mathbb{E}\sup_{0\leqslant u\leqslant T}\left|X^{(n)}(u)\right|^{\kappa}\right)du\notag\\
\leqslant& C_{\kappa}[(t-s)+(t-s)^{\frac{\kappa}{2}}],\notag
\end{align}
where $D$ is a positive number dependent of $\kappa$. Consequently, the desired assertion follows by substituting the above estimates of $B_1$, $B_2$, and $B_3$ into (\ref{timeDifference}) and taking the supremum arguments.
\end{proof}

With Lemma \ref{UBP}$\&$\ref{HolderC} at hand, we now show the following lemma. 
\begin{lemma}\label{Cauchy} (Cauchy sequences)
    The sequence $\{X^{(n)}(t)\}_{n\geqslant1}$ given by (\ref{Approximation-Eq}) is a Cauchy sequence in $L^{\kappa}(\Omega;D([0,T];\mathbb{R}^d))$ in the following sense: for any $n,m\geqslant1$, 
\begin{equation}\label{Cauchy-se}
    \left\|X^{(n)}-X^{(m)}\right\|_{L^{\kappa}}=\left(\mathbb{E}\left[\sup_{0\leqslant t\leqslant T}\left|X^{(n)}(t)-X^{(m)}(t)\right|^{\kappa}\right]\right)^{\frac{1}{\kappa}}\to0, \;\; \hbox{as}\; n, m\to\infty.
\end{equation}
\end{lemma}

\begin{proof}
Note that for $t\in[0,T]$,
\begin{align}
X^{(n)}(t)-X^{(m)}(t)=&\int_0^t\left[b\left(s,X^{(n)}(s),\mu^{(n)}_{[s]_{n}}\right)-b\left(s,X^{(m)}(s),\mu^{(m)}_{[s]_{m}}\right)\right]ds+\int_0^t\left[\sigma\left(s,X^{(n)}(s),\mu^{(n)}_{[s]_{n}}\right)-\sigma\left(s,X^{(m)}(s),\mu^{(m)}_{[s]_{m}}\right)\right]dW(s)\notag\\
&+\int_0^t\int_{U}\left[h\left(s,X^{(n)}(s),\mu^{(n)}_{[s]_{n}},z\right)-h\left(s,X^{(m)}(s),\mu^{(m)}_{[s]_{m}},z\right)\right]\tilde{N}(ds,dz).\notag
\end{align}
We define a stopping time
$$\tau_R:=\inf\left\{t\in[0,T]: \left|X^{(n)}(t)\right|\vee\left|X^{(m)}(t)\right|> R\right\},$$
for each $R>0$. It should be pointed out that the stopping time technique is applied here since (\ref{Approximation-Eq}) is a classical SDE (not distribution dependent). Then, by De Morgan's Law, we arrive at

\begin{equation}\label{JJ}
\mathbb{E}\left[\sup_{0\leqslant t\leqslant T}\left|X^{(n)}(t)-X^{(m)}(t)\right|^{\kappa}\right]=\mathbb{E}\Big[\sup_{0\leqslant t\leqslant T}\left|X^{(n)}(t)-X^{(m)}(t)\right|^{\kappa}\mathbb{I}_{\{\tau_R>T\}}\Big]+\mathbb{E}\Big[\sup_{0\leqslant t\leqslant T}\left|X^{(n)}(t)-X^{(m)}(t)\right|^{\kappa}\mathbb{I}_{\{\tau_R\leqslant T\}}\Big]:=J_1+J_2,
\end{equation}
where $\mathbb{I}_A$ is an indicator function of set $A$.

We next estimate each summand on the right hand side of equation above.

(1) Estimate of the term $J_1$. Notice that 
\begin{equation}\label{J_1}
J_1
\leqslant\mathbb{E}\left[\sup_{0\leqslant t\leqslant T}\left|X^{(n)}(t\wedge\tau_R)-X^{(m)}(t\wedge\tau_R)\right|^{\kappa}\right].
\end{equation}
By applying It\^o formula, we calculate that
$$
\left|X^{(n)}(t\wedge\tau_R)-X^{(m)}(t\wedge\tau_R)\right|^{\kappa}=J_{1,R}(t)+J_{2,R}(t)+J_{3,R}(t)+J_{4,R}(t)+J_{5,R}(t)+J_{6,R}(t),
$$
where 
\begin{align}
J_{1,R}(t)=&\kappa\int_0^{t\wedge\tau_R}\left|X^{(n)}(s)-X^{(m)}(s)\right|^{\kappa-2}\left\langle X^{(n)}(s)-X^{(m)}(s), b\left(s,X^{(n)}(s),\mu^{(n)}_{[s]_{n}}\right)-b\left(s,X^{(m)}(s),\mu^{(m)}_{[s]_{m}}\right)\right\rangle ds,\notag\\
J_{2,R}(t)=&\frac{\kappa}{2}\int_0^{t\wedge\tau_R}\left|X^{(n)}(s)-X^{(m)}(s)\right|^{\kappa-2}\left\| \sigma\left(s,X^{(n)}(s),\mu^{(n)}_{[s]_{n}}\right)-\sigma\left(s,X^{(m)}(s),\mu^{(m)}_{[s]_{m}}\right) \right\|^2 ds,\notag\\
J_{3,R}(t)=&\frac{\kappa(\kappa-2)}{2}\int_0^{t\wedge\tau_R}\left|X^{(n)}(s)-X^{(m)}(s)\right|^{\kappa-4}\left\| \left(X^{(n)}(s)-X^{(m)}(s)\right)^T\left(\sigma\left(s,X^{(n)}(s),\mu^{(n)}_{[s]_{n}}\right)-\sigma\left(s,X^{(m)}(s),\mu^{(m)}_{[s]_{m}}\right)\right) \right\|^2 ds,\notag\\
J_{4,R}(t)=&\int_0^{t}\int_U\bigg[\left|X^{(n)}(s)-X^{(m)}(s)+h\left(s,X^{(n)}(s),\mu^{(n)}_{[s]_n},z\right)-h\left(s,X^{(m)}(s),\mu^{(m)}_{[s]_{m}},z\right)\right|^{\kappa}-\left|X^{(n)}(s)-X^{(m)}(s)\right|^{\kappa}\notag\\
&-\kappa\left|X^{(n)}(s)-X^{(m)}(s)\right|^{\kappa-2}\left\langle X^{(n)}(s)-X^{(m)}(s), h\left(s,X^{(n)}(s),\mu^{(n)}_{[s]_n},z\right)-h\left(s,X^{(m)}(s),\mu^{(m)}_{[s]_{m}},z\right)\right\rangle \bigg] N(ds,dz),\notag\\
J_{5,R}(t)=&\kappa\int_0^{t\wedge\tau_R}\left|X^{(n)}(s)-X^{(m)}(s)\right|^{\kappa-2}\left\langle X^{(n)}(s)-X^{(m)}(s),\sigma\left(s,X^{(n)}(s),\mu^{(n)}_{[s]_{n}}\right)-\sigma\left(s,X^{(m)}(s),\mu^{(m)}_{[s]_{m}}\right)dW(s)\right\rangle,\notag\\
J_{6,R}(t)=&\kappa\int_0^{t\wedge\tau_R}\int_U\left|X^{(n)}(s)-X^{(m)}(s)\right|^{\kappa-2}\left\langle X^{(n)}(s)-X^{(m)}(s),h\left(s,X^{(n)}(s),\mu^{(n)}_{[s]_{n}},z\right)-h\left(s,X^{(m)}(s),\mu^{(m)}_{[s]_{m}},z\right)\right\rangle\tilde{N}(ds,dz).\notag
\end{align}
In order to take the suprema over the time and the expectation, we need to estimate $\mathbb{E}\big[\sup_{0\leqslant t\leqslant u}J_{i,R}(t)\big], i=1, \cdots, 6$, respectively. 

Note that $J_{i,R},i=1,2,3$ are standard Lebesgue integrals, we can estimate these three terms similarly. In fact, for any $u\in[0,T]$, by virtue of Assumptions {\bf A1-A2}, we obtain
\begin{align}
\mathbb{E}\left[\sup_{0\leqslant t\leqslant u}J_{1,R}(t)\right]\leqslant&\kappa\mathbb{E}\left[\sup_{0\leqslant t\leqslant u}\int_0^{t\wedge\tau_R}\left|X^{(n)}(s)-X^{(m)}(s)\right|^{\kappa-2}\left\langle X^{(n)}(s)-X^{(m)}(s), b\left(s,X^{(n)}(s),\mu^{(n)}_{[s]_{n}}\right)-b\left(s,X^{(m)}(s),\mu^{(n)}_{[s]_{n}}\right)\right\rangle
ds\right]\notag\\
&+\kappa\mathbb{E}\left[\int_0^{u\wedge\tau_R}\left|X^{(n)}(s)-X^{(m)}(s)\right|^{\kappa-1}\cdot\left|b\left(s,X^{(m)}(s),\mu^{(n)}_{[s]_{n}}\right)-b\left(s,X^{(m)}(s),\mu^{(m)}_{[s]_{m}}\right)\right|ds\right]\notag\\
\leqslant&\kappa L_R\mathbb{E}\left[\int_0^{u\wedge\tau_R}\left|X^{(n)}(s)-X^{(m)}(s)\right|^{\kappa}ds\right]+\kappa\sqrt{L}\mathbb{E}\left[\int_0^{u\wedge\tau_R}\left|X^{(n)}(s)-X^{(m)}(s)\right|^{\kappa-1}\cdot W_2\left(\mu^{(n)}_{[s]_{n}},\mu^{(m)}_{[s]_{m}}\right)ds\right],\notag
\end{align}
which on the application of Young's inequality (\ref{YE}) (with $\epsilon=1$, $p=\frac{\kappa}{\kappa-1}$, $q=\kappa$) yields
\begin{align}\label{J_{1,R}}
\mathbb{E}\left[\sup_{0\leqslant t\leqslant u}J_{1,R}(t)\right]\leqslant&\left(\kappa L_R+(\kappa-1)\sqrt{L}\right)\mathbb{E}\left[\int_0^{u\wedge\tau_R}\left|X^{(n)}(s)-X^{(m)}(s)\right|^{\kappa}ds\right]\notag\\
&+3^{\kappa-1}\sqrt{L}\int_0^{u\wedge\tau_R}\left[W_2^{\kappa}(\mu^{(n)}_{[s]_{n}},\mu^{(n)}_{s})+W_2^{\kappa}\left(\mu^{(n)}_{s},\mu^{(m)}_{s}\right)+W_2^{\kappa}\left(\mu^{(m)}_{s},\mu^{(m)}_{[s]_{m}}\right)\right]ds\notag\\
\leqslant&\left[\kappa L_R+\sqrt{L}\left(\kappa-1+3^{\kappa-1}\right)\right]\int_0^{u\wedge\tau_R}\mathbb{E}\left|X^{(n)}(s)-X^{(m)}(s)\right|^{\kappa}ds\notag\\
&+3^{\kappa-1}\sqrt{L}\int_0^{u\wedge\tau_R}\left(\mathbb{E}\left|X^{(n)}(s)-X^{(n)}([s]_{n})\right|^{\kappa}+\mathbb{E}\left|X^{(m)}(s)-X^{(m)}([s]_{m})\right|^{\kappa}\right)ds.
\end{align}
Analogously, by Assumptions {\bf A1-A2} and Young's inequality (\ref{YE}) (with $\epsilon=1$, $p=\frac{\kappa}{\kappa-2}$, $q=\frac{\kappa}{2}$) we can also obtain that
\begin{align}\label{J_{2,R}}
\mathbb{E}\left[\sup_{0\leqslant t\leqslant u}J_{2,R}(t)\right]
\leqslant&\left[\kappa L_R+L\left(\kappa-2+2\cdot3^{\kappa-1}\right)\right]\int_0^{u\wedge\tau_R}\mathbb{E}\left|X^{(n)}(s)-X^{(m)}(s)\right|^{\kappa}ds\notag\\
&+2L\cdot3^{\kappa-1}\int_0^{u\wedge\tau_R}\left(\mathbb{E}\left|X^{(n)}(s)-X^{(n)}([s]_{n})\right|^{\kappa}+\mathbb{E}\left|X^{(m)}(s)-X^{(m)}([s]_{m})\right|^{\kappa}\right)ds,
\end{align}
and 
\begin{align}\label{J_{3,R}}
\mathbb{E}\left[\sup_{0\leqslant t\leqslant u}J_{3,R}(t)\right]
\leqslant&(\kappa-2)\left[\kappa L_R+\left(\kappa-2+2\cdot3^{\kappa-1}\right)L\right]\int_0^{u\wedge\tau_R}\mathbb{E}\left|X^{(n)}(s)-X^{(m)}(s)\right|^{\kappa}ds\notag\\
&+2L\cdot3^{\kappa-1}(\kappa-2)\int_0^{u\wedge\tau_R}\left(\mathbb{E}\left|X^{(n)}(s)-X^{(n)}([s]_{n})\right|^{\kappa}+\mathbb{E}\left|X^{(m)}(s)-X^{(m)}([s]_{m})\right|^{\kappa}\right)ds.
\end{align}

As for the last three terms, we first use the remainder formula in (\ref{RF}) and Assumption {\bf A7} to get 
\begin{align}
\mathbb{E}\left[\sup_{0\leqslant t\leqslant u}J_{4,R}(t)\right]
\leqslant& C_1\mathbb{E}\int_0^{u\wedge\tau_R}\int_U\bigg(\left|X^{(n)}(s)-X^{(m)}(s)\right|^{\kappa-2}\left|h\left(s,X^{(n)}(s),\mu^{(n)}_{[s]_n},z\right)-h\left(s,X^{(m)}(s),\mu^{(m)}_{[s]_{m}},z\right)\right|^2\notag\\
&+\left|h\left(s,X^{(n)}(s),\mu^{(n)}_{[s]_n},z\right)-h\left(s,X^{(m)}(s),\mu^{(m)}_{[s]_{m}},z\right)\right|^{\kappa}\bigg)\nu(dz)ds.\notag\\
\leqslant&C_1\left(L_R+\frac{(\kappa-2+2\cdot3^{\kappa-1})L}{\kappa}\right)\int_0^{u\wedge\tau_R}\mathbb{E}\left|X^{(n)}(s)-X^{(m)}(s)\right|^{\kappa}ds\notag\\
&+2C_1L\cdot\frac{3^{\kappa-1}}{\kappa}\int_0^{u\wedge\tau_R}\left(\mathbb{E}\left|X^{(n)}(s)-X^{(n)}([s]_{n})\right|^{\kappa}+\mathbb{E}\left|X^{(m)}(s)-X^{(m)}([s]_{m})\right|^{\kappa}\right)ds\notag\\
&+C_1\mathbb{E}\int_0^{u\wedge\tau_R}\int_U\left|h\left(s,X^{(n)}(s),\mu^{(n)}_{[s]_n},z\right)-h\left(s,X^{(m)}(s),\mu^{(n)}_{[s]_{n}},z\right)\right|^{\kappa}\nu(dz)ds \notag\\
&+C_1\mathbb{E}\int_0^{u\wedge\tau_R}\int_U\left|h\left(s,X^{(m)}(s),\mu^{(n)}_{[s]_n},z\right)-h\left(s,X^{(m)}(s),\mu^{(m)}_{[s]_{m}},z\right)\right|^{\kappa}\nu(dz)ds\notag\\
\leqslant &C_1\left(L_R+\frac{(\kappa-2+2\cdot3^{\kappa-1})L}{\kappa}+L_R^{'}+L^{'}3^{\kappa-1}\right)\int_0^{u\wedge\tau_R}\mathbb{E}\left|X^{(n)}(s)-X^{(m)}(s)\right|^{\kappa}ds\notag\\
&+3^{\kappa-1}C_1\left(\frac{2L}{\kappa}+L^{'}\right)\int_0^{u\wedge\tau_R}\left(\mathbb{E}\left|X^{(n)}(s)-X^{(n)}([s]_{n})\right|^{\kappa}+\mathbb{E}\left|X^{(m)}(s)-X^{(m)}([s]_{m})\right|^{\kappa}\right)ds.
\end{align}
We exploit the Burkholder-Davis-Cundy inequality and Assumptions {\bf A2-A3} to obtain
\begin{align}
\mathbb{E}\left[\sup_{0\leqslant t\leqslant u}J_{5,R}(t)\right]\leqslant&\kappa\cdot\sqrt{32}\mathbb{E}\left(\int_0^{u\wedge\tau_R}\left|X^{(n)}(s)-X^{(m)}(s)\right|^{2\kappa-2}\cdot\left\|\sigma\left(s,X^{(n)}(s),\mu^{(n)}_{[s]_{n}}\right)-\sigma\left(s,X^{(m)}(s),\mu^{(m)}_{[s]_{m}}\right)\right\|^2ds\right)^{\frac{1}{2}}\notag\\
\leqslant&6\kappa\mathbb{E}\bigg[\sup_{0\leqslant t\leqslant u}\left|X^{(n)}(t\wedge\tau_R)-X^{(m)}(t\wedge\tau_R)\right|^{\kappa-1}\cdot\bigg(\int_0^{u\wedge\tau_R}2\left\|\sigma\left(s,X^{(n)}(s),\mu^{(n)}_{[s]_{n}}\right)-\sigma\left(s,X^{(m)}(s),\mu^{(n)}_{[s]_{n}}\right)\right\|^2\notag\\
&+2\left\|\sigma\left(s,X^{(m)}(s),\mu^{(n)}_{[s]_{n}}\right)-\sigma\left(s,X^{(m)}(s),\mu^{(m)}_{[s]_{m}}\right)\right\|^2ds\bigg)^{\frac{1}{2}}\bigg]\notag\\
\leqslant&6\kappa\mathbb{E}\left[\sup_{0\leqslant t\leqslant u}\left|X^{(n)}(t\wedge\tau_R)-X^{(m)}(t\wedge\tau_R)\right|^{\kappa-1}\cdot\left(\int_0^{u\wedge\tau_R}2L_R\left|X^{(n)}(s)-X^{(m)}(s)\right|^2+2LW_2^2\left(\mu^{(n)}_{[s]_{n}},\mu^{(m)}_{[s]_{m}}\right)ds\right)^{\frac{1}{2}}\right].\notag
\end{align}
 Then owing to Young's inequality (\ref{YE}) (with $\epsilon=\frac{1}{24(\kappa-1)}$, $p=\frac{\kappa}{\kappa-1}$, $q=\kappa$)
 , H\"older inequality,  elementary inequality and Lyapunov inequality, we further have
\begin{align}\label{J_{5,R}}
&\mathbb{E}\left[\sup_{0\leqslant t\leqslant u}J_{5,R}(t)\right]\notag\\
\leqslant& \frac{1}{4}\mathbb{E}\left[\sup_{0\leqslant t\leqslant u}\left|X^{(n)}(t\wedge\tau_R)-X^{(m)}(t\wedge\tau_R)\right|^{\kappa}\right]
+6^{\kappa}\cdot2^{\frac{\kappa}{2}}(4(\kappa-1))^{\kappa-1}\mathbb{E}\left(\int_0^{u\wedge\tau_R}L_R\left|X^{(n)}(s)-X^{(m)}(s)\right|^2+LW_2^2\left(\mu^{(n)}_{[s]_{n}},\mu^{(m)}_{[s]_{m}}\right)ds\right)^{\frac{\kappa}{2}}\notag\\
\leqslant &\frac{1}{4}\mathbb{E}\left[\sup_{0\leqslant t\leqslant u}\left|X^{(n)}(t\wedge\tau_R)-X^{(m)}(t\wedge\tau_R)\right|^{\kappa}\right]
+6^{\kappa}\cdot2^{\frac{\kappa}{2}}(4(\kappa-1))^{\kappa-1}u^{\frac{\kappa}{2}-1}\mathbb{E}\left[\int_0^{u\wedge\tau_R}\left(L_R\left|X^{(n)}(s)-X^{(m)}(s)\right|^2+LW_2^2\left(\mu^{(n)}_{[s]_{n}},\mu^{(m)}_{[s]_{m}}\right)\right)^{\frac{\kappa}{2}}ds\right]\notag\\
\leqslant & \frac{1}{4}\mathbb{E}\left[\sup_{0\leqslant t\leqslant u}\left|X^{(n)}(t\wedge\tau_R)-X^{(m)}(t\wedge\tau_R)\right|^{\kappa}\right]
+6^{\kappa}\cdot2^{\frac{\kappa}{2}}(4(\kappa-1))^{\kappa-1}u^{\frac{\kappa}{2}-1}\left(L_R^{\frac{\kappa}{2}}+L^{\frac{\kappa}{2}}3^{\kappa-1}\right)\int_0^{u\wedge\tau_R}\mathbb{E}\left|X^{(n)}(s)-X^{(m)}(s)\right|^{\kappa}ds\notag\\
&+6^{\kappa}\cdot(2L)^{\frac{\kappa}{2}}(12(\kappa-1))^{\kappa-1}u^{\frac{\kappa}{2}-1}\int_0^{u\wedge\tau_R}\left(\mathbb{E}\left|X^{(n)}(s)-X^{(n)}([s]_{n})\right|^{\kappa}+\mathbb{E}\left|X^{(m)}(s)-X^{(m)}([s]_{m})\right|^{\kappa}\right)ds.
\end{align}
Finally, we apply Kunita's first inequality (see, e.g., Lemma 2.1 of \cite{Dareiotis2016}) and Young's inequality (\ref{YE}) (with $\epsilon=\frac{1}{4D(\kappa-1)}$, $p=\frac{\kappa}{\kappa-1}$, $q=\kappa$) to obtain
\begin{align}\label{J_{6,R}}
&\mathbb{E}\left[\sup_{0\leqslant t\leqslant u}J_{6,R}(t)\right]\notag\\
\leqslant&\kappa D\mathbb{E}\left(\int_0^{u\wedge\tau_R}\int_U\left|X^{(n)}(s)-X^{(m)}(s)\right|^{2\kappa-2}\cdot\left|h\left(s,X^{(n)}(s),\mu^{(n)}_{[s]_{n}},z\right)-h\left(s,X^{(m)}(s),\mu^{(m)}_{[s]_{m}},z\right)\right|^2\nu(dz)ds\right)^{\frac{1}{2}}\notag\\
\leqslant & \frac{1}{4}\mathbb{E}\left[\sup_{0\leqslant t\leqslant u}\left|X^{(n)}(t\wedge\tau_R)-X^{(m)}(t\wedge\tau_R)\right|^{\kappa}\right]
+D^{\kappa}\cdot2^{\frac{\kappa}{2}}(4(\kappa-1))^{\kappa-1}u^{\frac{\kappa}{2}-1}\left(L_R^{\frac{\kappa}{2}}+L^{\frac{\kappa}{2}}3^{\kappa-1}\right)\int_0^{u\wedge\tau_R}\mathbb{E}\left|X^{(n)}(s)-X^{(m)}(s)\right|^{\kappa}ds\notag\\
&+D^{\kappa}\cdot(2L)^{\frac{\kappa}{2}}(12(\kappa-1))^{\kappa-1}u^{\frac{\kappa}{2}-1}\int_0^{u\wedge\tau_R}\left(\mathbb{E}\left|X^{(n)}(s)-X^{(n)}([s]_{n})\right|^{\kappa}+\mathbb{E}\left|X^{(m)}(s)-X^{(m)}([s]_{m})\right|^{\kappa}\right)ds.
\end{align}
Substituting (\ref{J_{1,R}})-(\ref{J_{6,R}}) into (\ref{J_1}) yields that 
\begin{align}
&\mathbb{E}\left[\sup_{0\leqslant t\leqslant u}\left|X^{(n)}(t\wedge\tau_R)-X^{(m)}(t\wedge\tau_R)\right|^2\right]\notag\\
\leqslant&2\Bigg[\kappa^2 L_R+(\kappa-1+3^{\kappa-1})\sqrt{L}+(\kappa-1)(\kappa-2+2\cdot3^{\kappa-1})L+C_1\left(L_R+\frac{(\kappa-2+2\cdot3^{\kappa-1})L}{\kappa}+L_R^{'}+L^{'}3^{\kappa-1}\right)\notag\\
&+(6^{\kappa}+D^{\kappa})2^{\frac{\kappa}{2}}(4(\kappa-1))^{\kappa-1}u^{\frac{\kappa}{2}-1}\left(L_R^{\frac{\kappa}{2}}+L^{\frac{\kappa}{2}}3^{\kappa-1}\right)\Bigg]\int_0^{u}\mathbb{E}\sup_{0\leqslant t\leqslant s}\left|X^{(n)}(t\wedge\tau_R)-X^{(m)}(t\wedge\tau_R)\right|^{\kappa}ds\notag\\
&+2\cdot3^{\kappa-1}\left[\sqrt{L}+2L(\kappa-1)+C_1\left(\frac{2L}{\kappa}+L^{'}\right)+(6^{\kappa}+D^{\kappa})(2L)^{\frac{\kappa}{2}}(4(\kappa-1))^{\kappa-1}u^{\frac{\kappa}{2}-1}\right]\notag\\
&\cdot\int_0^{u\wedge\tau_R}\left(\sup_{0\leqslant t\leqslant s}\mathbb{E}\left|X^{(n)}(t)-X^{(n)}([t]_{n})\right|^{\kappa}+\sup_{0\leqslant t\leqslant s}\mathbb{E}\left|X^{(m)}(t)-X^{(m)}([t]_{m})\right|^{\kappa}\right)ds.\notag
\end{align}
In addition, for any $t\in[0,T]$, the result in Lemma \ref{HolderC} implies that 
$$\mathbb{E}\left|X^{(n)}(t)-X^{(n)}([t]_{n})\right|^{\kappa}\leqslant Ch_n \;\; \hbox{and} \;\; \mathbb{E}\left|X^{(m)}(t)-X^{(m)}([t]_{n})\right|^{\kappa}\leqslant Ch_m.$$
By the above estimates, together with Gr\"onwall inequality, we conclude that 
\begin{align}\label{JJ-1}
J_1\leqslant&\mathbb{E}\left[\sup_{0\leqslant t\leqslant T}\left|X^{(n)}(t\wedge\tau_R)-X^{(m)}(t\wedge\tau_R)\right|^{\kappa}\right]\notag\\
\leqslant&2CT(h_n+h_m)\cdot3^{\kappa-1}\left[\sqrt{L}+2L(\kappa-1)+C_1\left(\frac{2L}{\kappa}+L^{'}\right)+(6^{\kappa}+D^{\kappa})(2L)^{\frac{\kappa}{2}}(4(\kappa-1))^{\kappa-1}T^{\frac{\kappa}{2}-1}\right]\notag\\
&\cdot e^{2T\left[\kappa^2 L_R+(\kappa-1+3^{\kappa-1})\sqrt{L}+(\kappa-1)(\kappa-2+2\cdot3^{\kappa-1})L+C_1\left(L_R+\frac{(\kappa-2+2\cdot3^{\kappa-1})L}{\kappa}+L_R^{'}+L^{'}3^{\kappa-1}\right)+(6^{\kappa}+D^{\kappa})2^{\frac{\kappa}{2}}(4(\kappa-1))^{\kappa-1}T^{\frac{\kappa}{2}-1}\left(L_R^{\frac{\kappa}{2}}+L^{\frac{\kappa}{2}}3^{\kappa-1}\right)\right]}\notag\\
:&=T(h_n+h_m)C(\kappa,T,L,L^{'})e^{T\cdot C(\kappa,T,L,L_R,L^{'},L_R^{'})}.
\end{align}

(2) Estimate of the term $J_2$. With aid of the Cauchy-Schwarz inequality, one can have
\begin{align}\label{J-2}
J_2=\mathbb{E}\left[\sup_{0\leqslant t\leqslant T}\left|X^{(n)}(t)-X^{(m)}(t)\right|^{\kappa}\mathbb{I}_{\{\tau_R\leqslant T\}}\right]
\leqslant&\sqrt{\mathbb{E}\left(\sup_{0\leqslant t\leqslant T}\left|X^{(n)}(t)-X^{(m)}(t)\right|^{\kappa}\right)^2}\sqrt{\mathbb{E}\left(\mathbb{I}_{\{\tau_R\leqslant T\}}\right)^2}\notag\\
\leqslant&2^{\kappa-\frac{1}{2}}\sqrt{\mathbb{E}\left[\sup_{0\leqslant t\leqslant T}\left|X^{(n)}(t)\right|^{2\kappa}+\sup_{0\leqslant t\leqslant T}\left|X^{(m)}(t)\right|^{2\kappa}\right]}\sqrt{\mathbb{P}\Big(\tau_R\leqslant T\Big)}\notag\\
\leqslant&C_1\sqrt{\mathbb{P}\Big(\tau_R\leqslant T\Big)}.
\end{align}
Here we have used the result of Lemma \ref{UBP}. Notice that, applying the subadditivity of probability and using Lemma \ref{UBP} again, we can estimate that 
\begin{equation*}
\mathbb{P}\Big(\tau_R\leqslant T\Big)
\leqslant\mathbb{E}\left(\mathbb{I}_{\{\tau_R\leqslant T\}}\frac{|X^{(n)}(\tau_R)|^{4}+|X^{(m)}(\tau_R)|^{4}}{R^4}\right)\leqslant\frac{1}{R^4}\left(\mathbb{E}\sup_{0\leqslant t\leqslant T}\left|X^{(n)}(t)\right|^4+\mathbb{E}\sup_{0\leqslant t\leqslant T}\left|X^{(m)}(t)\right|^4\right)\leqslant \frac{C}{R^4}.
\end{equation*}
By substituting this into (\ref{J-2}), we further obtain 
\begin{equation}\label{JJ-2}
J_2\leqslant \frac{C}{R^2}.
\end{equation}

At this point, we can estimate \eqref{JJ} by combining \eqref{JJ-1} and \eqref{JJ-2}:
\begin{equation}\label{JJ-1+2}
\mathbb{E}\sup_{0\leqslant t\leqslant T}\left|X^{(n)}(t)-X^{(m)}(t)\right|^\kappa\leqslant T(h_n+h_m)C(\kappa,T,L,L^{'})e^{T\cdot C(\kappa,T,L,L_R,L^{'},L_R^{'})}
+\frac{C}{R^2}.
\end{equation}
Note that $R$ is independent of $n$, $m$,  and $\frac{C}{R^2}$ converges to $0$ as $R\to\infty$. For any given $\varepsilon>0$, there exists a sufficiently large number $R(\varepsilon)>0$, such that, 
$$\frac{C}{R_{*}^2}<\frac{\varepsilon}{2},$$
when $R_{*}\geqslant R(\varepsilon)$. Since both $h_n$ and $h_m$ converge to 0 as $n,m\to\infty$, with the $\varepsilon>0$ chosen above, we have 
$$
T(h_n+h_m)C(\kappa,T,L,L^{'})e^{T\cdot C(\kappa,T,L,L_{R_*},L^{'},L_{R^{'}_*})}<\frac{\varepsilon}{2},
$$
by letting $n,m\to\infty$. Consequently, we conclude that \eqref{Cauchy-se} holds.
\end{proof}

\subsection{Proof of Theorem \ref{mainresult1}}


In this subsection, we turn to proving our main theorem in this section. The proof consists of three steps.

{\bf Step one: (Existence)} Let $\{X^{(n)}(t)\}_{n\geqslant1}$ be a Cauchy sequence in $L^{\kappa}(\Omega;D([0,T];\mathbb{R}^d))$ given by \eqref{Approximation-Eq}. Keep in mind that $L^{\kappa}(\Omega;D([0,T];\mathbb{R}^d))$ is a complete space under the $L^{\kappa}$ norm. Thus there exists an $\{\mathcal{F}_t\}_{0\leqslant t\leqslant T}$-adapted $\mathbb{R}^d$-valued c\`{a}dl\`{a}g stochastic process $\{X(t)\}_{0\leqslant t\leqslant T}$ with $X(0)=x_0$ and $\mu_t=\mathcal{L}_{X(t)}$ such that
\begin{equation}\label{L2C}
\lim_{n\to \infty}\left[\mathbb{E}\left|X^{(n)}(t)-X(t)\right|^{\kappa}\right]^{\frac{1}{\kappa}}\leqslant\lim_{n\to \infty}\left[\mathbb{E}\sup_{0\leqslant t\leqslant T}\left|X^{(n)}(t)-X(t)\right|^{\kappa}\right]^{\frac{1}{\kappa}}=0.
\end{equation}
We next prove that $\{X(t)\}_{0\leqslant t\leqslant T}$ is a solution to \eqref{Main-equation}. Indeed, the main idea is to show that the right-hand side of (\ref{Approximation-Eq}) converges in probability to
$$x_0+\int_0^tb(s,X(s),\mu_s)ds+\int_0^t\sigma(s,X(s),\mu_s)dW(s)+\int_0^t\int_Uh(s,X(s),\mu_s,z)\tilde{N}(dz,ds),$$
by taking the limit on both sides of (\ref{Approximation-Eq}). Here $\mu_s=\mathcal{L}(X(s))$ for any $s\in[0,T]$.

First of all, it follows from (\ref{L2C}) that, there exists a subsequence (for notational simplicity, still indexed by $n$) such that, for any $s\in[0,T]$, 
$$
X^{(n)}(s,\omega) \to X(s,\omega),\;\; \mathbb{P}\text{-a.s.}
$$
By Lemma \ref{HolderC}, the Wasserstein metric of $\mu^{(n)}_{[s]_n}$ and $\mu_s$ satisfies 
\begin{align}\label{md}
\lim_{n\to \infty}\sup_{0\leqslant s\leqslant t}W_{2}^{\kappa}\left(\mu^{(n)}_{[s]_n},\mu_s\right)&\leqslant2^{\kappa-1}\lim_{n\to \infty}\sup_{0\leqslant s\leqslant t}\mathbb{E}\left|X^{(n)}(s)-X^{(n)}([s]_n)\right|^{\kappa}+2^{\kappa-1}\lim_{n\to \infty}\mathbb{E}\left[\sup_{0\leqslant s\leqslant t}\left|X^{(n)}(s)-X(s)\right|^{\kappa}\right]\notag\\
&\leqslant2^{\kappa-1}C\lim_{n\to\infty}h_n=0.
\end{align}
Taking Assumption {\bf A6} into account, we immediately have that, for any $s\in[0,T]$ and almost all $\omega\in\Omega$, 
\begin{equation} 
b\left(s, X^{(n)}(s), \mu^{(n)}_{[s]_n}\right)\to b\left(s, X(s),\mu_s\right),\;
\sigma\left(s, X^{(n)}(s), \mu^{(n)}_{[s]_n}\right)\to \sigma\left(s, X(s),\mu_s\right),\notag
\end{equation}
\begin{equation}\label{Conver-1}
\int_Uh\left(s, X^{(n)}(s), \mu^{(n)}_{[s]_n},z\right)\nu(dz)\to\int_Uh\left(s,X(s),\mu_s,z\right)\nu(dz),
\end{equation}
as $n\to\infty$.

Next, we claim that $\{b(s, X^{(n)}(s), \mu^{(n)}_{[s]_n})\}_{n\geqslant1}$ and $\{\sigma(s, X^{(n)}(s), \mu^{(n)}_{[s]_n})\}_{n\geqslant1}$ are uniformly integrable. In fact, from 
Assumptions {\bf A4-A5} and Lemma \ref{UBP}, we obtain the following uniform boundedness: 
\begin{align}
    &\sup_{n\geqslant1}\mathbb{E}|b(s, X^{(n)}(s), \mu^{(n)}_{[s]_n})|
    \leqslant\sqrt{3K_1}\sup_{n\geqslant1}\mathbb{E}[1+|X^{(n)}(s)|^{\frac{\kappa}{2}}+W_2^{\frac{\kappa}{2}}( \mu^{(n)}_{[s]_n},\delta_0)]\leqslant \sqrt{3K_1}(1+2C), \notag\\
    &\sup_{n\geqslant1}\mathbb{E}\|\sigma(t, X^{(n)}(s), \mu^{(n)}_{[s]_n})\|^2
    \leqslant
    K\sup_{n\geqslant1}\mathbb{E}[1+|X^{(n)}(s)|^{2}+W_2^{2}( \mu^{(n)}_{[s]_n},\delta_0)]\leqslant K(1+2C),\notag
    \end{align}
    and the following uniform absolute continuity: 
    \begin{align}
    &\sup_{n\geqslant1}\mathbb{E}\left[\left|b\left(s, X^{(n)}(s), \mu^{(n)}_{[s]_n}\right)\right|\cdot \mathbb{I}_A\right]
    \leqslant K_1\sup_{n\geqslant1}\left[\mathbb{E}\left(1+\left|X^{(n)}(s)\right|^{\kappa}+W_2^{\kappa}\left( \mu^{(n)}_{[s]_n},\delta_0\right)\right)\right]^{\frac{1}{2}}(\mathbb{P}(A))^{\frac{1}{2}}\leqslant K_1\sqrt{1+2C}(\mathbb{P}(A))^{\frac{1}{2}}\to 0,\notag\\
    &\sup_{n\geqslant1}\mathbb{E}\left[\left\|\sigma\left(s, X^{(n)}(s), \mu^{(n)}_{[s]_n}\right)\right\|^2\cdot \mathbb{I}_A\right]
    \leqslant K\sup_{n\geqslant1}\left[\mathbb{E}\left(1+\left|X^{(n)}(s)\right|^{2}+W_2^{2}\left( \mu^{(n)}_{[s]_n},\delta_0\right)\right)^2\right]^{\frac{1}{2}}(\mathbb{P}(A))^{\frac{1}{2}}\leqslant K\sqrt{3(1+2C)}(\mathbb{P}(A))^{\frac{1}{2}}\to 0,\notag
\end{align}
when $\mathbb{P}(A)\to 0$.
Here, the elementary inequality 
\begin{equation}
\left(\sum_{i=1}^k|a_i|\right)^l\leqslant\left(k\max_{1\leqslant i\leqslant k}|a_i|\right)^l\leqslant k^l\sum_{i=1}^k|a_i|^l, \; \forall l>0, \; a_i\in\mathbb{R}, k\in\mathbb{N}, \notag
\end{equation}
has been used. The uniform integrability for $\{b(s, X^{(n)}(s), \mu^{(n)}_{[s]_n})\}_{n\geqslant1}$ as well as $\{\sigma(s, X^{(n)}(s), \mu^{(n)}_{[s]_n})\}_{n\geqslant1}$ follows by Lemma 3, page 190 of \cite{Shiryave1989}. 

Hence, the dominated convergence theorem (see, e.g., Theorem 4, page 188 of \cite{Shiryave1989}), together with (\ref{Conver-1}) yields, for any $s\in[0,T]$,
\begin{equation}\label{Eb1}
\lim_{n\to \infty}\mathbb{E}\left|b\left(s, X^{(n)}(s), \mu^{(n)}_{[s]_n}\right)-b\left(s, X(s),\mu_s\right)\right|=0,
\end{equation}
\begin{equation}\label{Esg1}
\lim_{n\to \infty}\mathbb{E}\left\|\sigma\left(s, X^{(n)}(s), \mu^{(n)}_{[s]_n}\right)-\sigma\left(s, X(s),\mu_s\right)\right\|^2=0.
\end{equation}

In addition, note that, following from (\ref{L2C}),
\begin{equation}
\mathbb{E}\left[\sup_{0\leqslant t\leqslant T}\left|X(t)\right|^{\kappa}\right]\leqslant C.\notag
\end{equation}
We further have the following estimates based on Assumptions {\bf A4-A5} and Lemma \ref{UBP},
\begin{align}\label{Eb2}
\sup_{n\geqslant1}\sup_{s\in[0,t]}\mathbb{E}\left|b\left(s, X^{(n)}(s), \mu^{(n)}_{[s]_n}\right)-b(s, X(s),\mu_s)\right|&\leqslant \sqrt{3K_1}\sup_{n\geqslant1}\sup_{s\in[0,t]}\mathbb{E}\left[2+\left|X^{(n)}(s)\right|^{\frac{\kappa}{2}}+|X(s)|^{\frac{\kappa}{2}}+W_2^{\frac{\kappa}{2}}\left( \mu^{(n)}_{[s]_n},\delta_0\right)+W_2^{\frac{\kappa}{2}}(\mu_s,\delta_0)\right]\notag\\
&\leqslant 2\sqrt{3K_1}(1+2C),
\end{align}
\begin{align}\label{Esg2}
\sup_{n\geqslant1}\sup_{s\in[0,t]}\mathbb{E}\left\|\sigma\left(s, X^{(n)}(s), \mu^{(n)}_{[s]_n}\right)-\sigma(s, X(s),\mu_s)\right\|^2&\leqslant 2K\sup_{n\geqslant1}\sup_{s\in[0,t]}\mathbb{E}\left[2+\left|X^{(n)}(s)\right|^{2}+|X(s)|^{2}+W_2^{\kappa}\left( \mu^{(n)}_{[s]_n},\delta_0\right)+W_2^{2}(\mu_t,\delta_0)\right]\notag\\
&\leqslant 4K(1+2C).
\end{align}
For any $t\in[0,T]$, by the dominated convergence combined with (\ref{Eb1}) and (\ref{Eb2}), we eventually obtain 
\begin{align}\label{Step-1-1}
\lim_{n\to \infty}\mathbb{E}\left|\int_0^t\left(b\left(s, X^{(n)}(s), \mu^{(n)}_{[s]_n}\right)-b(s, X(s),\mu_s)\right)ds\right|\leqslant\lim_{n\to \infty}\int_0^t\mathbb{E}\left|b\left(s, X^{(n)}(s), \mu^{(n)}_{[s]_n}\right)-b(s, X(s),\mu_s)\right|ds=0.
\end{align}
Similarly, in view of (\ref{Esg1}) and (\ref{Esg2}), we arrive at
\begin{align}\label{Step-1-2}
\lim_{n\to \infty}\mathbb{E}\left|\int_0^t\left(\sigma\left(s, X^{(n)}(s), \mu^{(n)}_{[s]_n}\right)-\sigma(s, X(s),\mu_s)\right)dW(s)\right|^2=\lim_{n\to \infty}\int_0^t\mathbb{E}\left\|\sigma\left(s, X^{(n)}(s), \mu^{(n)}_{[s]_n}\right)-\sigma(s, X(s),\mu_s)\right\|^2ds
=0.
\end{align}

Finally, we look into the estimates for the integral with respect to the Poisson
random measure. For any $u\in[0,T]$, it follows from Kunita's first inequality  and Assumptions {\bf A1-A2} that
\begin{align}
&\mathbb{E}\sup_{0\leqslant u\leqslant t}\left|\int_0^u\int_U\left(h\left(s,X^{(n)}(s),\mu^{(n)}_{[s]_n},z\right)-h(s,X(s),\mu_s,z)\right)\tilde{N}(ds,dz)\right|\notag\\
\leqslant& D\mathbb{E}\left[\int_0^t\int_U\left|h\left(s,X^{(n)}(s),\mu^{(n)}_{[s]_n},z\right)-h(s,X(s),\mu_s,z)\right|^2\nu(dz)ds\right]^{\frac{1}{2}}\notag\\
\leqslant&\sqrt{2}D\mathbb{E}\left[\int_0^t\int_U\left(\left|h\left(s,X^{(n)}(s),\mu^{(n)}_{[s]_n},z\right)-h(s,X(s),\mu^{(n)}_{[s]_n},z)\right|^2+\left|h\left(s,X(s),\mu^{(n)}_{[s]_n},z\right)-h(s,X(s),\mu_s,z)\right|^2\right)\nu(dz)ds\right]^{\frac{1}{2}}\notag\\
\leqslant&\sqrt{2}D\mathbb{E}\left[\int_0^t\left|X^{(n)}(s)-X(s)\right|^2ds+\int_0^tW_2^{2}\left(\mu^{(n)}_{[s]_n},\mu_s\right)ds\right]^{\frac{1}{2}}\notag\\
\leqslant&2D\mathbb{E}\left[\int_0^t\left|X^{(n)}(s)-X(s)\right|^2ds\right]^{\frac{1}{2}}+2D\mathbb{E}\left[\int_0^tW_2^{2}\left(\mu^{(n)}_{[s]_n},\mu_s\right)ds\right]^{\frac{1}{2}}\notag\\
\leqslant&2D\sqrt{T}\mathbb{E}\left[\sup_{0\leqslant s\leqslant t}\left|X^{(n)}(s)-X(s)\right|\right]+2D\sqrt{T}\sup_{0\leqslant s\leqslant t}W_2\left(\mu^{(n)}_{[s]_n},\mu_s\right).\notag\
\end{align}
By (\ref{L2C}), (\ref{md}) and Lyapunov inequality, we conclude that 
\begin{align}\label{Step-1-3}
\lim_{n\to\infty}\mathbb{E}\sup_{0\leqslant u\leqslant t}\left|\int_0^u\int_U\left(h\left(s,X^{(n)}(s),\mu^{(n)}_{[s]_n},z\right)-h(s,X(s),\mu_s,z)\right)\tilde{N}(ds,dz)\right|=0.
\end{align}

As a consequence, by \eqref{Step-1-1}, \eqref{Step-1-2} and \eqref{Step-1-3}, $\{X(t)\}_{0\leqslant t\leqslant T}$ is a strong solution to (\ref{Main-equation}). We are done with the existence.

{\bf Step two: (Boundedness).} For $t\in[0,T]$, let $X(t)\in L^{\kappa}(\Omega;D([0,T];\mathbb{R}^d))$ be any solution to equation (\ref{Main-equation}). In what follows, we calculate the $r$th moment of the solution $(X(t))_{0\leqslant t\leqslant T}$.

For every $R>0$, we define the stopping time 
$$\pi_R:=\inf\big\{t\in[0,T]:|X(t)|> R\big\}\wedge T.$$
It is clear that $|X(t)|\leqslant R$ for $0\leqslant t\leqslant \pi_R$. Using the same technique as Lemma \ref{UBP} was proved,  we calculate that for any $u\in[0,T]$,
 \begin{align}
&\mathbb{E}\sup_{0\leqslant t\leqslant  u\wedge\pi_R}|X(t)|^r\notag\\
\leqslant&\mathbb{E}|x_0|^r+\frac{1}{2}\mathbb{E}\sup_{0\leqslant t\leqslant u\wedge\pi_R}|X(t)|^r+K(r-2)\frac{r^2+r+2C_1}{2r}\int_0^u\left(1+\mathbb{E}\sup_{0\leqslant t\leqslant s\wedge\pi_R}|X(t)|^r\right)ds\notag\\
&+\left[2\cdot3^{\frac{r}{2}-1}(r+1)K+C_2^r2^{r}(r-1)^{r-1}(2K)^{\frac{r}{2}}(3T)^{\frac{r}{2}-1}+2C_1\cdot\left(\frac{2K}{r}3^{\frac{r}{2}-1}+K_2\right)\right]\int_0^u\left(1+\mathbb{E}\sup_{0\leqslant t\leqslant s\wedge\pi_R}|X(t)|^r\right)ds.\notag
\end{align} 
Notice that $\pi_R\to T,$ $\mathbb{P}$-a.s. Therefore, we can finish the proof of the estimate (\ref{UBeq}) by Gr\"onwall  inequality and the Fatou lemma. That is
 $$\mathbb{E}\sup_{0\leqslant t\leqslant T}|X(t)|^r\leqslant \liminf_{R\to\infty}\mathbb{E}\sup_{0\leqslant t\leqslant T\wedge\pi_R}|X(t)|^r\leqslant C_r<\infty.$$

{\bf Step three: (Uniqueness)} Let $X(t)$, $Y(t)$ be two solutions for  (\ref{Main-equation}) on the same probability space with $X(0)=Y(0)$. By (\ref{UBeq}), for a fixed $r\geqslant\max\{\kappa^2/2, 4\}$, there exists positive constant $C_r$ such that
$$\mathbb{E}\left[\sup_{0\leqslant t\leqslant T}|X(t)|^r\right]\leqslant C_r,\; \mathbb{E}\left[\sup_{0\leqslant t\leqslant T}|Y(t)|^r\right]\leqslant C_r.$$
For a sufficiently large $R>0$, we define the stopping time
$$\bar{\tau}_R:=\inf\big\{t\in[0,T]: |X(t)|\vee|Y(t)|\geqslant R\big\}.$$
Let us now replace $|X^{(n)}(t)-X^{(m)}(t)|$ and $\tau_R$ by $|X(t)-Y(t)|$ and $\bar{\tau}_R$, respectively. Then, applying the same proof process as the existence of solution to equation (\ref{Main-equation}) was proved, we have
\begin{align}
\mathbb{E}\left[\sup_{0\leqslant t\leqslant T}|X(t)-Y(t)|^{2}\right]&=\mathbb{E}\left[\sup_{0\leqslant t\leqslant T}|X(t)-Y(t)|^{2}\mathbb{I}_{\{\bar{\tau}_R>T\}}\right]+\mathbb{E}\left[\sup_{0\leqslant t\leqslant T}|X(t)-Y(t)|^{2}\mathbb{I}_{\{\bar{\tau}_R\leqslant T\}}\right]\notag\\
&\leqslant\mathbb{E}\left[\sup_{0\leqslant t\leqslant T}\left|X(t\wedge\bar{\tau}_R)-Y(t\wedge\bar{\tau}_R)\right|^{2}\right]+C_1\sqrt{\mathbb{P}\Big(\bar{\tau}_R\leqslant T\Big)}\leqslant\frac{C}{R^2}.\notag
\end{align}
Letting $R\to\infty$, it gives the uniqueness of solution to equation (\ref{Main-equation}).

This completes the proof of Theorem \ref{mainresult1}.

%

\renewcommand{\theequation}{\thesection.\arabic{equation}}
\setcounter{equation}{0}

\section{Stochastic averaging principle}\label{sec3}
In this section, we establish a stochastic averaging principle for the following stochastic integral equation
\begin{align} \label{AMain}
 X_{\varepsilon}(t)=&x_0+\int_0^t b\left(\frac{s}{\varepsilon },X_{\varepsilon}(s),\mathscr{L}_{X_{\varepsilon}(s)}\right)ds+\int_0^t\sigma\left(\frac{s}{\varepsilon },X_{\varepsilon}(s),\mathscr{L}_{X_{\varepsilon}(s)}\right)dW(s)\notag\\
  &+\int_0^t\int_{U}h\left(\frac{s}{\varepsilon },X_{\varepsilon}(s),\mathscr{L}_{X_{\varepsilon}(s)},z\right)\tilde{N}(ds,dz),\;\;\;t\in[0,T],
\end{align}
where $\varepsilon$ is a positive small parameter ($0<\varepsilon\ll1$). Assuming \eqref{AMain} fulfills the conditions in Assumptions {\bf A1-A7}, the existence and uniqueness of its solution follows immediately from Theorem \ref{mainresult1}. 

As mentioned in the Introduction, our main objection is to show that the solution (i.e., $X^{\varepsilon}(t)$, $t\in[0,T]$) of \eqref{AMain} could be approximated by some other simpler (or averaged) process in an appropriate sense. To proceed, 
we associate (\ref{AMain}) with the following averaged McKean-Vlasov SDE:
\begin{align} \label{ASDE}
 \bar{X}(t)=x_0+\int_0^t \bar{b}\left(\bar{X}(s),\mathscr{L}_{\bar{X}(s)}\right)ds+\int_0^t\bar{\sigma}\left(\bar{X}(s),\mathscr{L}_{\bar{X}(s)}\right)dW(s)+\int_0^t\int_{U}\bar{h}\left(\bar{X}(s),\mathscr{L}_{\bar{X}(s)},z\right)\tilde{N}(ds,dz),\;\;\;t\in[0,T],
\end{align}
where $\bar{b}: \mathbb{R}^d\times M_2(\mathbb{R}^d)\to \mathbb{R}^d$, $\bar{\sigma}: \mathbb{R}^d\times M_2(\mathbb{R}^{d})\to\mathbb{R}^{d\times m}$ and $\bar{h}: \mathbb{R}^d\times M_2(\mathbb{R}^d)\times U\to \mathbb{R}^d$ are Borel measurable functions. To ensure that (\ref{ASDE}) also admits a unique solution and to apply the stochastic averaging arguments, we need to make use of the following averaging conditions. We point out that such conditions are slightly different from the classical conditions (see, e.g., \cite{Xu2011,Shen2022}) due to the type of nonlinearity terms.

\noindent{\bf A8. } (Averaging conditions) There exist positive bounded functions (sometimes known as rate functions of convergence)  $\varphi_i(t)$ ($t\in[0,T]$) with $\lim_{t\to\infty}\varphi_i(t)=0$, $i=1, 2, 3$, such that 
$$\frac{1}{t}\int_{0}^{t}|b(s,x,\mu)-\bar{b}(x,\mu)|^2ds\leqslant\varphi_1(t)C_R^b(1+|x|^2),\;\;\;\; \frac{1}{t}\int_{0}^{t}\|\sigma(s,x,\mu)-\bar{\sigma}(x,\mu)\|^2ds\leqslant\varphi_2(t)C_R^{\sigma}(1+|x|^2),$$
$$\frac{1}{t}\int_{0}^{t}\int_{U}|h(s,x,\mu,z)-\bar{h}(x,\mu,z)|^2\nu(dz)ds\leqslant\varphi_3(t)C_R^{h}(1+|x|^2),$$
respectively, for any $t\in[0,T]$, $x,y\in\mathbb{R}^d$ with $|x|\vee|y|\leqslant R$, and $\mu\in\mathcal{M}_2(\mathbb{R}^d)$.\par

Furthermore, if $\kappa>2$, we need an additional condition.\\
\noindent{\bf A9. }(Additional averaging conditions w.r.t. the jump coefficients) There exist a positive bounded function $\varphi(t)$ ($t\in[0,T]$) with $\lim_{t\to\infty}\varphi(t)=0$ such that 
$$\frac{1}{t}\int_{0}^{t}\int_{U}|h(s,x,\mu,z)-\bar{h}(x,\mu,z)|^r\nu(dz)ds\leqslant\varphi(t)C_R^{h}(1+|x|^r),$$
$$\frac{1}{t}\int_{0}^{t}\int_{U}|h(s,x,\mu,z)-\bar{h}(x,\mu,z)|^{\kappa}\nu(dz)ds\leqslant\varphi(t)C_R^{h}(1+|x|^{\kappa}),$$
respectively, for any $t\in[0,T]$, $x,y\in\mathbb{R}^d$ with $|x|\vee|y|\leqslant R$, and $\mu\in\mathcal{M}_2(\mathbb{R}^d)$.

The main theorem on an averaging principle for (\ref{AMain}) is thus formulated below.
\begin{theorem}\label{Th} {\bf (An Averaging Principle)}
Suppose that Assumptions {\bf A1-A9} hold. Then, we have the following averaging principle
\begin{equation}
\lim_{\varepsilon\to0}\mathbb{E}\sup_{0\leqslant t\leqslant T}|X_{\varepsilon}(t)-\bar{X}(t)|^{2}=0.
\end{equation}
\end{theorem}

Moreover, by Theorem \ref{Th} and Chebyshev-Markov inequality, we have the following corollary.

\begin{corollary}
The original solution $X_{\varepsilon}(t)$ converges in probability to the averaged solution $\bar{X}(t)$, that is, for any given number $\delta>0$,
\begin{equation*}
\mathbb{P}(\sup_{0\leqslant t\leqslant T}\left|X_{\varepsilon}(t)-\bar{X}(t)\right|>\delta)\to 0,
\end{equation*}
as $\varepsilon\to 0$.
\end{corollary}


Prior to prove Theorem \ref{Th}, the well-posedness of equation (\ref{ASDE}) should be considered. The following lemma shows that it can be assured.

\begin{lemma}\label{lemma3-1}
There is a unique solution $\bar{X}(t)$ to the averaged equation (\ref{ASDE}) under Assumptions {\bf A1-A9}.
\end{lemma}

\begin{proof}
By Theorem \ref{mainresult1}, we only need to check that coefficients functions $\bar{b}$, $\bar{\sigma}$, $\bar{h}$ fulfill the conditions for the existence and uniqueness of the solution. Note that both \eqref{AMain} and (\ref{ASDE}) have the same initial value $x_0$. The condition in Assumption {\bf A6} holds directly. For the conditions in Assumptions {\bf A1-A5}, we only discuss about the case of function $\bar{b}$, and the results for functions $\bar{\sigma}, \bar{h}$ follow by similar arguments. At last, we verify that $\bar{h}$ satisfies the condition in Assumption {\bf A7}. The details of this proof are given in the Appendix.
\end{proof}

We next complete the proof of Theorem \ref{Th} in what follows.
\begin{proof}[Proof of Theorem \ref{Th}] 

For any $t\in[0,T]$, it follows from (\ref{AMain}) and (\ref{ASDE}) that 
\begin{align}
X_{\varepsilon}(t)-\bar{X}(t)=&\int_0^t\left[b\left(\frac{s}{\varepsilon},X_{\varepsilon}(s),\mathscr{L}_{X_{\varepsilon}(s)}\right)-\bar{b}\left(\bar{X}(s),\mathscr{L}_{\bar{X}(s)}\right)\right]ds+\int_0^t\left[\sigma\left(\frac{s}{\varepsilon},X_{\varepsilon}(s),\mathscr{L}_{X_{\varepsilon}(s)}\right)-\bar{\sigma}\left(\bar{X}(s),\mathscr{L}_{\bar{X}(s)}\right)\right]dW(s)\notag\\
&+\int_0^t\int_{U}\left[h\left(\frac{s}{\varepsilon},X_{\varepsilon}(s),\mathscr{L}_{X_{\varepsilon}(s)},z\right)-\bar{h}\left(\bar{X}(s),\mathscr{L}_{\bar{X}(s)},z\right)\right]\tilde{N}(ds,dz).\notag
\end{align}
To deal with the one-sided local Lipschitz case, for each $R>0$, a stopping time is defined as follows:
$$\eta_R:=\inf\big\{t\in[0,T]: |X_{\varepsilon}(t)|\vee|\bar{X}(t)|> R\big\}.$$
According to De Morgan's Law, we have
\begin{equation*}
\mathbb{E}\left[\sup_{0\leqslant t\leqslant T}|X_{\varepsilon}(t)-\bar{X}(t)|^{2}\right]
=\mathbb{E}\left[\sup_{0\leqslant t\leqslant T}|X_{\varepsilon}(t)-\bar{X}(t)|^{2}\mathbb{I}_{\{\eta_R>T\}}\right]+\mathbb{E}\left[\sup_{0\leqslant t\leqslant T}|X_{\varepsilon}(t)-\bar{X}(t)|^{2}\mathbb{I}_{\{\eta_R\leqslant T\}}\right]:=I_1+I_2.
\end{equation*}
We now calculate each term on the right hand side of equation above.

(1) Calculation of the term $I_1$. We first obtain that 
\begin{equation}\label{I_1}
I_1=\mathbb{E}\left[\sup_{0\leqslant t\leqslant T}|X_{\varepsilon}(t)-\bar{X}(t)|^{2}\mathbb{I}_{\{\eta_R>T\}}\right]\leqslant\mathbb{E}\left[\sup_{0\leqslant t\leqslant T}|X_{\varepsilon}(t\wedge\eta_R)-\bar{X}(t\wedge\eta_R)|^{2}\right].
\end{equation}
By using the It\^o formula, we have
$$|X_{\varepsilon}(t\wedge\eta_R)-\bar{X}(t\wedge\tau_R)|^{2}=\sum_{i=1}^5\Lambda_{i}(t),$$
where
\begin{align}
    \Lambda_{1}(t)=&2\int_0^{t\wedge\eta_R}\left\langle X_{\varepsilon}(s)-\bar{X}(s), b\left(\frac{s}{\varepsilon},X_{\varepsilon}(s),\mathscr{L}_{X_{\varepsilon}(s)}\right)-\bar{b}\left(\bar{X}(s),\mathscr{L}_{\bar{X}(s)}\right)\right\rangle ds,\notag\\
    \Lambda_{2}(t)=&\int_0^{t\wedge\eta_R}\left\| \sigma\left(\frac{s}{\varepsilon},X_{\varepsilon}(s),\mathscr{L}_{X_{\varepsilon}(s)}\right)-\bar{\sigma}\left(\bar{X}(s),\mathscr{L}_{\bar{X}(s)}\right)\right\|^2 ds,\notag\\
    \Lambda_{3}(t)=&\int_0^{t\wedge\eta_R}\int_U\left|h\left(\frac{s}{\varepsilon},X_{\varepsilon}(s),\mathscr{L}_{X_{\varepsilon}(s)},z\right)-\bar{h}\left(\bar{X}(s),\mathscr{L}_{\bar{X}(s)},z\right)\right|^2 N(ds,dz),\notag\\
    \Lambda_{4}(t)=&2\int_0^{t\wedge\eta_R}\left\langle X_{\varepsilon}(s)-\bar{X}(s), \sigma\left(\frac{s}{\varepsilon},X_{\varepsilon}(s),\mathscr{L}_{X_{\varepsilon}(s)}\right)-\bar{\sigma}\left(\bar{X}(s),\mathscr{L}_{\bar{X}(s)}\right)dW(s)\right\rangle,\notag\\
    \Lambda_{5}(t)=&2\int_0^{t\wedge\eta_R}\left\langle X_{\varepsilon}(s)-\bar{X}(s),h\left(\frac{s}{\varepsilon},X_{\varepsilon}(s),\mathscr{L}_{X_{\varepsilon}(s)},z\right)-\bar{h}\left(\bar{X}(s),\mathscr{L}_{\bar{X}(s)},z\right)\right\rangle\tilde{N}(ds,dz).\notag
\end{align}
By taking supremum over $[0, u]$ for $u\in[0,T]$ and then taking expectations, we next estimate $\mathbb{E}\big[\sup_{0\leqslant t\leqslant u}\Lambda_{i}(t)\big], i=1, \ldots, 5$, respectively. In view of of Assumptions {\bf A1, A2, A8}, we obtain
\begin{align}\label{I_{1,R}}
\mathbb{E}\left[\sup_{0\leqslant t\leqslant u}\Lambda_{1}(t)\right]\leqslant&2\mathbb{E}\left[\sup_{0\leqslant t\leqslant u}\int_0^{t\wedge\eta_R}\left\langle X_{\varepsilon}(s)-\bar{X}(s), b\left(\frac{s}{\varepsilon},X_{\varepsilon}(s),\mathscr{L}_{X_{\varepsilon}(s)}\right)-b\left(\frac{s}{\varepsilon},\bar{X}(s),\mathscr{L}_{X_{\varepsilon}(s)}\right)\right\rangle ds\right]\notag\\
&+2\mathbb{E}\left[\sup_{0\leqslant t\leqslant u}\int_0^{t\wedge\eta_R}|X_{\varepsilon}(s)-\bar{X}(s)|\cdot\left|b\left(\frac{s}{\varepsilon},\bar{X}(s),\mathscr{L}_{X_{\varepsilon}(s)}\right)-b\left(\frac{s}{\varepsilon},\bar{X}(s),\mathscr{L}_{\bar{X}(s)}\right)\right|ds\right]\notag\\
&+2\mathbb{E}\left[\sup_{0\leqslant t\leqslant u}\int_0^{t\wedge\eta_R}|X_{\varepsilon}(s)-\bar{X}(s)|\cdot\left|b\left(\frac{s}{\varepsilon},\bar{X}(s),\mathscr{L}_{\bar{X}(s)}\right)-\bar{b}\left(\bar{X}(s),\mathscr{L}_{\bar{X}(s)}\right)\right|ds\right]\notag\\
\leqslant&2L_R\mathbb{E}\left[\int_0^{u\wedge\eta_R}|X_{\varepsilon}(s)-\bar{X}(s)|^2ds\right]+2\sqrt{L}\mathbb{E}\left[\int_0^{u\wedge\eta_R}|X_{\varepsilon}(s)-\bar{X}(s)|\cdot W_2\left(\mathscr{L}_{X_{\varepsilon}(s)},\mathscr{L}_{\bar{X}(s)}\right)ds\right]\notag\\
&+2\mathbb{E}\left[\int_0^{u\wedge\eta_R}|X_{\varepsilon}(s)-\bar{X}(s)|\cdot\left|b\left(\frac{s}{\varepsilon},\bar{X}(s),\mathscr{L}_{\bar{X}(s)}\right)-\bar{b}\left(\bar{X}(s),\mathscr{L}_{\bar{X}(s)}\right)\right|ds\right]\notag\\
\leqslant&(2L_R+2\sqrt{L}+1)\mathbb{E}\left[\int_0^{u\wedge\eta_R}|X_{\varepsilon}(s)-\bar{X}(s)|^2ds\right]\notag\\
&+u\mathbb{E}\left[\frac{\varepsilon}{u\wedge\eta_R}\int_0^{\frac{u\wedge\eta_R}{\varepsilon}}\left|b\left(s,\bar{X}(s\varepsilon),\mathscr{L}_{\bar{X}(s\varepsilon)}\right)-\bar{b}\left(\bar{X}(s\varepsilon),\mathscr{L}_{\bar{X}(s\varepsilon)}\right)\right|^2ds\right]\notag\\
\leqslant&(2L_R+2\sqrt{L}+1)\int_0^u\mathbb{E}\sup_{0\leqslant t\leqslant s}|X_{\varepsilon}(t\wedge\eta_R)-\bar{X}(t\wedge\eta_R)|^2ds+uC_R^b\varphi_1\left(\frac{u\wedge\eta_R}{\varepsilon}\right)\left(1+\mathbb{E}\sup_{0\leqslant t\leqslant u}|\bar{X}(t)|^2\right)\notag\\
\leqslant&(2L_R+2\sqrt{L}+1)\int_0^u\mathbb{E}\sup_{0\leqslant t\leqslant s}|X_{\varepsilon}(t\wedge\eta_R)-\bar{X}(t\wedge\eta_R)|^2ds+uC_R^b\cdot C\varphi_1\left(\frac{u\wedge\eta_R}{\varepsilon}\right).
\end{align}
Here we have used the fact that, for each $u\in[0,T]$, $\left(\mathbb{E}\sup_{0\leqslant t\leqslant u}|\bar{X}(t)|^2\right)^{\frac{1}{2}}\leqslant(\mathbb{E}\sup_{0\leqslant t\leqslant u}|\bar{X}(t)|^r)^{\frac{1}{r}}<\infty$ if $\mathbb{E}|X_{\varepsilon}(0)|^r<\infty$. 
By Assumptions {\bf A1, A2, A8}, and using the same techniques as (\ref{I_{1,R}})  was proved, we get
\begin{align}\label{I_{2,R}}
\mathbb{E}\left[\sup_{0\leqslant t\leqslant u}\Lambda_{2}(t)\right]\leqslant&3\mathbb{E}\left[\sup_{0\leqslant t\leqslant u}\int_0^{t\wedge\eta_R}\left\|\sigma\left(\frac{s}{\varepsilon},X_{\varepsilon}(s),\mathscr{L}_{X_{\varepsilon}(s)}\right)-\sigma\left(\frac{s}{\varepsilon},\bar{X}(s),\mathscr{L}_{X^{\varepsilon}(s)}\right)\right\|^2 ds\right]\notag\\
&+3\mathbb{E}\left[\sup_{0\leqslant t\leqslant u}\int_0^{t\wedge\eta_R}\left\|\sigma\left(\frac{s}{\varepsilon},\bar{X}(s),\mathscr{L}_{X_{\varepsilon}(s)}\right)-\sigma\left(\frac{s}{\varepsilon},\bar{X}(s),\mathscr{L}_{\bar{X}(s)}\right)\right\|^2 ds\right]\notag\\
&+3\mathbb{E}\left[\sup_{0\leqslant t\leqslant u}\int_0^{t\wedge\eta_R}\left\|\sigma\left(\frac{s}{\varepsilon},\bar{X}(s),\mathscr{L}_{\bar{X}(s)}\right)-\bar{\sigma}\left(\bar{X}(s),\mathscr{L}_{\bar{X}(s)}\right)\right\|^2 ds\right]\notag\\
\leqslant&3L_R\mathbb{E}\left[\int_0^{u\wedge\eta_R}|X_{\varepsilon}(s)-\bar{X}(s)|^2ds\right]+3L\mathbb{E}\left[\int_0^{u\wedge\eta_R}W_2^2\left(\mathscr{L}_{X_{\varepsilon}(s)},\mathscr{L}_{\bar{X}(s)}\right)ds\right]\notag\\
&+3\mathbb{E}\left[\int_0^{u\wedge\eta_R}\left\|\sigma\left(\frac{s}{\varepsilon},\bar{X}(s),\mathscr{L}_{\bar{X}(s)}\right)-\bar{\sigma}\left(\bar{X}(s),\mathscr{L}_{\bar{X}(s)}\right)\right\|^2 ds\right]\notag\\
\leqslant&3(L_R+L)\int_0^u\mathbb{E}\sup_{0\leqslant t\leqslant s}|X_{\varepsilon}(t\wedge\eta_R)-\bar{X}(t\wedge\eta_R)|^2ds+3uC_R^{\sigma}\cdot C\varphi_2\left(\frac{u\wedge\eta_R}{\varepsilon}\right).
\end{align}
By a similar argument as (\ref{I_{2,R}}), one uses Assumptions {\bf A1, A2, A8} to get
\begin{align}\label{I_{3,R}}
\mathbb{E}\left[\sup_{0\leqslant t\leqslant u}\Lambda_{3}(t)\right]&\leqslant\mathbb{E}\left[\int_0^{t\wedge\eta_R}\int_U\left|h\left(\frac{s}{\varepsilon},X_{\varepsilon}(s),\mathscr{L}_{X_{\varepsilon}(s)},z\right)-\bar{h}\left(X(s),\mathscr{L}_{\bar{X}(s)},z\right)\right|^2\nu(dz)ds\right] \notag\\
&\leqslant3(L_R+L)\int_0^u\mathbb{E}\sup_{0\leqslant t\leqslant s}|X_{\varepsilon}(t\wedge\eta_R)-\bar{X}(t\wedge\eta_R)|^2ds+3uC_R^{h}\cdot C\varphi_3\left(\frac{u\wedge\eta_R}{\varepsilon}\right).
\end{align}
Due to Burkholder-Davis-Cundy inequality, Young's inequality \eqref{YE} (with $\epsilon=\frac{1}{24},p=q=2$), and (\ref{I_{2,R}}), we have 
\begin{align}\label{I_{4,R}}
\mathbb{E}\left[\sup_{0\leqslant t\leqslant u}\Lambda_{4}(t)\right]\leqslant&2\cdot\sqrt{32}\mathbb{E}\left(\int_0^{u\wedge\eta_R}|X_{\varepsilon}(s)-\bar{X}(s)|^2\cdot\left\|\sigma\left(\frac{s}{\varepsilon},X_{\varepsilon}(s),\mathscr{L}_{X_{\varepsilon}(s)}\right)-\bar{\sigma}\left(\bar{X}(s),\mathscr{L}_{\bar{X}(s)}\right)\right\|^2ds\right)^{\frac{1}{2}}\notag\\
\leqslant&12\mathbb{E}\left[\sup_{0\leqslant t\leqslant u}|X_{\varepsilon}(t\wedge\eta_R)-\bar{X}(t\wedge\eta_R)|\cdot\left(\int_0^{u\wedge\eta_R}\left\|\sigma\left(\frac{s}{\varepsilon},X_{\varepsilon}(s),\mathscr{L}_{X_{\varepsilon}(s)}\right)-\bar{\sigma}\left(\bar{X}(s),\mathscr{L}_{\bar{X}(s)}\right)\right\|^2ds\right)^{\frac{1}{2}}\right]\notag\\
\leqslant&\frac{1}{4}\mathbb{E}\left[\sup_{0\leqslant t\leqslant u}|X_{\varepsilon}(t\wedge\eta_R)-\bar{X}(t\wedge\eta_R)|^2\right]+144\mathbb{E}\left[\int_0^{u\wedge\eta_R}\left\|\sigma\left(\frac{s}{\varepsilon},X_{\varepsilon}(s),\mathscr{L}_{X_{\varepsilon}(s)}\right)-\bar{\sigma}\left(\bar{X}(s),\mathscr{L}_{\bar{X}(s)}\right)\right\|^2ds\right]\notag\\
\leqslant&\frac{1}{4}\mathbb{E}\left[\sup_{0\leqslant t\leqslant u}|X_{\varepsilon}(t\wedge\eta_R)-\bar{X}(t\wedge\eta_R)|^2\right]+432(L_R+L)\int_0^u\mathbb{E}\sup_{0\leqslant t\leqslant s}|X_{\varepsilon}(t\wedge\eta_R)-\bar{X}(t\wedge\eta_R)|^2ds\notag\\
&+432uC_R^{\sigma}\cdot C\varphi_2\left(\frac{u\wedge\eta_R}{\varepsilon}\right).
\end{align}
Analogously, in view of the Kunita's first inequality, Young's inequality \eqref{YE} (with $\epsilon=\frac{1}{4D},p=q=2$) and (\ref{I_{3,R}}), we arrive at
\begin{align}\label{I_{5,R}}
\mathbb{E}\left[\sup_{0\leqslant t\leqslant u}\Lambda_{5}(t)\right]\leqslant&2D\mathbb{E}\left(\int_0^{u\wedge\eta_R}\int_U|X_{\varepsilon}(s)-\bar{X}(s)|^2\cdot\left|h\left(\frac{s}{\varepsilon},X_{\varepsilon}(s),\mathscr{L}_{X^{\varepsilon}(s)},z\right)-\bar{h}\left(\bar{X}(s),\mathscr{L}_{\bar{X}(s)},z\right)\right|^2\nu(dz)ds\right)^{\frac{1}{2}}\notag\\
\leqslant&\frac{1}{4}\mathbb{E}\left[\sup_{0\leqslant t\leqslant u}|X_{\varepsilon}(t\wedge\eta_R)-\bar{X}(t\wedge\eta_R)|^2\right]\notag\\
&+4D^2\mathbb{E}\left[\int_0^{u\wedge\eta_R}\int_U\left|h\left(\frac{s}{\varepsilon},X_{\varepsilon}(s),\mathscr{L}_{X_{\varepsilon}(s)},z\right)-\bar{h}\left(\bar{X}(s),\mathscr{L}_{\bar{X}(s)},z\right)\right|^2\nu(dz)ds\right]\notag\\
\leqslant&\frac{1}{4}\mathbb{E}\left[\sup_{0\leqslant t\leqslant u}|X_{\varepsilon}(t\wedge\eta_R)-\bar{X}(t\wedge\eta_R)|^2\right]+12D^2(L_R+L)\int_0^u\mathbb{E}\sup_{0\leqslant t\leqslant s}|X_{\varepsilon}(t\wedge\eta_R)-\bar{X}(t\wedge\eta_R)|^2ds\notag\\
&+12D^2uC_R^{h}\cdot C\varphi_3\left(\frac{u\wedge\eta_R}{\varepsilon}\right).
\end{align}
By substituting (\ref{I_{1,R}})-(\ref{I_{5,R}}) into (\ref{I_1}), and further 
 utilizing the Gr\"onwall inequality, we obtain
\begin{align}\label{II-1}
I_1\leqslant&\mathbb{E}\left[\sup_{0\leqslant t\leqslant T}|X_{\varepsilon}(t\wedge\eta_R)-\bar{X}(t\wedge\eta_R)|^2\right]\notag\\
\leqslant&\left[2TC_R^b\cdot C\varphi_1\left(\frac{T\wedge\eta_R}{\varepsilon}\right)+870TC_R^{\sigma}\cdot C\varphi_2\left(\frac{T\wedge\eta_R}{\varepsilon}\right)+6(1+4D^2)TC_R^{\sigma}C\varphi_3\left(\frac{T\wedge\eta_R}{\varepsilon}\right)\right]\notag\\
&\cdot e^{\left[4(L_R+\sqrt{L})+2+12\cdot(73+2D^2)(L_R+L^2)\right]T}.
\end{align}

(2) Calculation of the term $I_2$. Using the Cauchy-Schwarz inequality and Theorem \ref{mainresult1}, we deduce that
\begin{align}\label{II-2}
I_2=&\mathbb{E}\left[\sup_{0\leqslant t\leqslant T}|X_{\varepsilon}(t)-\bar{X}(t)|^2\mathbb{I}_{\{\eta_R\leqslant T\}}\right]
\leqslant\sqrt{\mathbb{E}\left(\sup_{0\leqslant t\leqslant T}|X_{\varepsilon}(t)-\bar{X}(t)|^2\right)^2}\sqrt{\mathbb{E}\left(\mathbb{I}_{\{\eta_R\leqslant T\}}\right)^2}\notag\\
\leqslant&2\sqrt{2}\sqrt{\mathbb{E}\left(\sup_{0\leqslant t\leqslant T}|X_{\varepsilon}(t)|^4+\sup_{0\leqslant t\leqslant T}|\bar{X}(t)|^4\right)}\sqrt{\mathbb{E}\left(\mathbb{I}_{\{\eta_R\leqslant T\}}\frac{|X_{\varepsilon}(\eta_R)|^4+|\bar{X}(\eta_R)|^4}{R^4}\right)}\notag\\
\leqslant&\frac{2\sqrt{2}}{R^2}\left(\mathbb{E}\sup_{0\leqslant t\leqslant T}|X_{\varepsilon}(t)|^4+\mathbb{E}\sup_{0\leqslant t\leqslant T}|\bar{X}(t)|^4\right)\leqslant \frac{C}{R^2}.
\end{align}

By combining the estimates of $I_1$ and $I_2$, i.e., \eqref{II-1} and \eqref{II-2}, we conclude that
\begin{align}
\mathbb{E}\sup_{0\leqslant t\leqslant T}|X_{\varepsilon}(t)-\bar{X}(t)|^2\leqslant& \left[2TC_R^b\cdot C\varphi_1\left(\frac{T\wedge\eta_R}{\varepsilon}\right)+870TC_R^{\sigma}\cdot C\varphi_2\left(\frac{T\wedge\eta_R}{\varepsilon}\right)+6(1+4D^2)TC_R^{\sigma}C\varphi_3\left(\frac{T\wedge\eta_R}{\varepsilon}\right)\right]\notag\\
&\cdot e^{\left[4(L_R+\sqrt{L})+2+12\cdot(73+2D^2)(L_R+L^2)\right]T}+\frac{C}{R^2}.\notag
\end{align}
Now, for any $\delta>0$, we can choose $R>0$ sufficiently large such that $\frac{C}{R^4}<\frac{\delta}{2}$. Furthermore, by taking $\varepsilon$ small enough and using the averaging condition {\bf A8}, we can have that
$$\left[2TC_R^b\cdot C\varphi_1\left(\frac{T\wedge\eta_R}{\varepsilon}\right)+870TC_R^{\sigma}\cdot C\varphi_2\left(\frac{T\wedge\eta_R}{\varepsilon}\right)+6(1+4D^2)TC_R^{\sigma}C\varphi_3\left(\frac{T\wedge\eta_R}{\varepsilon}\right)\right]\cdot e^{\left[4(L_R+\sqrt{L})+2+12\cdot(73+2D^2)(L_R+L^2)\right]T}<\frac{\delta}{2}.$$
Therefore, the arbitrariness of $\delta$ implies that $\mathbb{E}\sup_{0\leqslant t\leqslant T}|X_{\varepsilon}(t)-\bar{X}(t)|^2$ converges to 0, as $\varepsilon$ goes to 0. This completes the proof.
\end{proof}

\renewcommand{\theequation}{\thesection.\arabic{equation}}
\setcounter{equation}{0}

\section{Example}\label{sec4}
In this section, let us present an illustrative example for our theoretical results of this paper.
\\
{\bf Example 4.1.} Consider the following one-dimensional McKean-Vlasov SDE
\begin{align}\label{Example-equation}
dX_{\varepsilon}(t)=&\left[\left(X_{\varepsilon}(t)-X_{\varepsilon}^3(t)\right)\frac{\frac{t}{\varepsilon}}{1+\frac{t}{\varepsilon}}+\mathbb{E}X_{\varepsilon}(t)\right]dt+\left[X_{\varepsilon}(t)\sin\left(\log^2\left(1+X_{\varepsilon}^2(t)\right)\right)\frac{\frac{t}{\varepsilon}}{2+\frac{t}{\varepsilon}}+\mathbb{E}X_{\varepsilon}(t)\right]dW(t) \notag\\
&+\int_{U}\left[X_{\varepsilon}(t)\sin\left(\log^{\frac{3}{2}}\left(1+X_{\varepsilon}^2(t)\right)\right)\left(1-e^{-\frac{t}{\varepsilon}}\right)+\mathbb{E}X_{\varepsilon}(t)\right]\tilde{N}(dt, dz), \;\;\; t\in[0,T],
\end{align}
with $X_{\varepsilon}(0)=x_0$ being a constant. Here, $W(t)$ is a scalar Wiener process,  $U=\mathbb{R}\backslash \{ 0\}$, and $\nu$ is a finite measure with $\nu(U)=1$. Define 
$$b(t,x,\mu)=(x-x^3)\frac{t}{1+t}+\int_{\mathbb{R}}y\mu(dy), \;\; \sigma(t,x,\mu)=\psi(x)\frac{t}{2+t}+\int_{\mathbb{R}}y\mu(dy),\;\; h(t,x,\mu,z)=\phi(x)(1-e^{-t})+\int_{\mathbb{R}}y\mu(dy),$$
where $\psi(x)=x\sin(\log^2(1+x^2))$ and $\phi(x)=x\sin(\log^{\frac{3}{2}}(1+x^2))$ are two continuously differentiable functions. For any $x\in \mathbb{R}$, we can calculate that
\begin{equation}\label{Twoab}
|\psi(x)|\leqslant|x|, \;\;|\phi(x)|\leqslant |x|,\;\;|(\partial_x\psi)(x)|\leqslant 1+4\log(1+x^2) \;\;\text{and}\;\;|(\partial_x\phi)(x)|\leqslant 1+3\sqrt{\log(1+x^2)}.
\end{equation}

{\bf (1) Well-posedness.} To show that (\ref{Example-equation}) has a unique solution $(X_{\varepsilon}(t))_{0\leqslant t\leqslant T}$, we need to examine that the conditions in Theorem \ref{mainresult1} are satisfied. For any $R>0$, $x,y\in\mathbb{R}$  with $|x|\vee|y|\leqslant R$ and $\mu\in\mathcal{M}_2(\mathbb{R})$, one computes
\begin{align}
(x-y)(b(t,x,\mu)-b(t,y,\mu))&=(x-y)(x-x^3-y+y^3)\frac{t}{1+t}\leqslant |x-y|^2-|x-y|^2(x^2+xy+y^2)\notag\\
&\leqslant |x-y|^2(1-xy)\leqslant (1+R^2)|x-y|^2:=L_R^1|x-y|^2,\notag
\end{align}
\begin{align}
|\sigma(t,x,\mu)-\sigma(t,y,\mu)|^2&=\left|\psi(x)-\psi(y)\right|^2\left(\frac{t}{2+t}\right)^2\leqslant \left|\int_0^1(\partial_x\psi)(y+\theta(x-y))(x-y)d\theta\right|^2\notag\\
&\leqslant \left[\sup_{|z|\leqslant R}\left|(\partial_x\psi)(z)\right|\right]^2\cdot|x-y|^2\leqslant \left[\sup_{|z|\leqslant R}\left(1+4\log(1+z^2)\right)\right]^2\cdot|x-y|^2\notag\\
&\leqslant \left(1+4\log(1+R^2)\right)^2|x-y|^2 := L_R^2|x-y|^2,\notag
\end{align}
\begin{align}
\int_U|h(t,x,\mu,z)-h(t,y,\mu,z)|^2\nu(dz)&=\left|\phi(x)-\phi(y)\right|^2(1-e^{-t})^2\leqslant \left|\int_0^1(\partial_x\phi)(y+\theta(x-y))(x-y)d\theta\right|^2\notag\\
&\leqslant \left[\sup_{|z|\leqslant R}|(\partial_x\phi)(z)|\right]^2\cdot|x-y|^2\leqslant \left[\sup_{|z|\leqslant R}\left(1+3\sqrt{\log(1+z^2)}\right)\right]^2\cdot|x-y|^2\notag\\
&\leqslant \left(1+3\sqrt{\log(1+R^2)}\right)^2|x-y|^2 := L_R^3|x-y|^2.\notag
\end{align}
Denote $L_R=\max\{L_R^1,L_R^2,L_R^3\}$. It implies that Assumption {\bf A1} is satisfied. According to the expressions of $b$, $\sigma$ and $h$, for any $x\in\mathbb{R}$ and $\mu_1,\mu_2\in\mathcal{M}_2(\mathbb{R})$, we calculate
\begin{align}
&|b(t,x,\mu_1)-b(t,x,\mu_2)|^2+\|\sigma(t,x,\mu_1)-\sigma(t,x,\mu_2)\|^2+\int_U\left|h(t,x,\mu_1,z)-h(t,y,\mu_2,z)\right|^2\nu(dz)\notag\\
=&3\left|\int_{\mathbb{R}}y\mu_1(dy)-\int_{\mathbb{R}}y\mu_2(dy)\right|^2\leqslant3W_2^2(\mu_1,\mu_2),\notag
\end{align}
which means that $b$, $\sigma$ and $h$ meet Assumptions {\bf A2-A3}. In addition, owing to (\ref{Twoab}), the boundeness of $\frac{t}{1+t}$, $\frac{t}{2+t}$ and $1-e^{-t}$, we can deduce that, for any $x\in\mathbb{R}^d$, $\mu\in\mathcal{M}_2(\mathbb{R}),
$
 \begin{align}
x\cdot b(t,x,\mu)&\leqslant x(x-x^3)\left(\frac{t}{1+t}\right)+x\int_{\mathbb{R}}y\mu(dy)\leqslant x^2+\frac{1}{2}x^2+\frac{1}{2}\left(\int_{\mathbb{R}}y\mu(dy)\right)^2\leqslant 2\left(1+x^2+W_2^{2}(\mu,\delta_0)\right), \notag
\end{align}
\begin{align}
|\sigma(t,x,\mu)|^2=\left|\psi(x)\frac{t}{2+t}+\int_{\mathbb{R}}y\mu(dy)\right|^2\leqslant2|\psi(x)|^2+2\left(\int_{\mathbb{R}}y\mu(dy)\right)^2
\leqslant 2\left(1+|x|^2+W_2^2(\mu,\delta_0)\right),\notag
\end{align}
\begin{align}
\int_U|h(t,x,\mu,z)|^2\nu(dz)=\left|\phi(x)(1-e^{-t})+\int_{\mathbb{R}}y\mu(dy)\right|^2\leqslant2|\phi(x)|^2+2\left(\int_{\mathbb{R}}y\mu(dy)\right)^2
\leqslant 2\left(1+|x|^2+W_2^2(\mu,\delta_0)\right).\notag
\end{align}
Thus Assumption {\bf A4} holds.
 
Note that, for any $x\in\mathbb{R}^d$, $\mu\in\mathcal{M}_2(\mathbb{R}),
$
 \begin{align}
|b(t,x,\mu)|^2\leqslant2(x-x^3)^2\left(\frac{t}{1+t}\right)^2+2\left(\int_{\mathbb{R}}y\mu(dy)\right)^2\leqslant 2x^2-4x^4+2x^6+2W_2^2(\mu,\delta_0)\leqslant 4\left(1+x^6+W_2^{6}(\mu,\delta_0)\right). \notag
\end{align}
We can find that Assumption {\bf A5} holds with $\kappa=6$. 

In addition, note that $X_{\varepsilon}(0)$ is a constant, thus Assumption {\bf A6} (with $r\geqslant 18$) naturally holds. Due to the expression of $h$ and the finiteness of $\nu$, Assumption {\bf A7} can be verified easily using the same technique as Assumptions {\bf A1-A2} was checked.

{\bf (2) The averaging principle.} Define $$\bar{b}(x,\mu)=x-x^3+\int_{\mathbb{R}}y\mu(dy), \;\;\bar{\sigma}(x,\mu)=\psi(x)+\int_{\mathbb{R}}y\mu(dy),
\;\;
\bar{h}(x,\mu,z)=\phi(x)+\int_{\mathbb{R}}y\mu(dy).$$ Then we can verify that the averaging conditions in Assumption {\bf A8-A9} (with $\kappa=6$ and $r\geqslant 18$) are satisfied:
\begin{align}
\frac{1}{t}\int_{0}^{t}|b(s,x,\mu)-\bar{b}(x,\mu)|^2ds=\frac{1}{t}\int_{0}^{t}|x-x^3|^2\left[1-\frac{s}{1+s}\right]^2ds=x^2(1-x^2)^2\frac{1}{1+t}\leqslant \varphi_1(t)C_R^b(1+|x|^2)\notag
\end{align}
\begin{align}
\frac{1}{t}\int_{0}^{t}|\sigma(s,x,\mu)-\bar{\sigma}(x,\mu)|^2ds=\frac{1}{t}\int_{0}^{t}\psi^2(x)\left[1-\frac{s}{2+s}\right]^2ds=\psi^2(x)\frac{2(1+t)}{(2+t)t}\leqslant \varphi_2(t)C_R^{\sigma}(1+|x|^2)\notag
\end{align}
\begin{align}
\frac{1}{t}\int_{0}^{t}\int_U|h(s,x,\mu,z)-\bar{h}(x,\mu,z)|^2\nu(dz)ds=\frac{1}{t}\int_{0}^{t}\phi^2(x)[1-(1-e^{-s})]^2ds=\phi^2(x)\frac{1-e^{-2t}}{2t}\leqslant \varphi_3(t)C_R^{h}(1+|x|^2),\notag
\end{align}
and
\begin{align}
\frac{1}{t}\int_{0}^{t}\int_U|h(s,x,\mu,z)-\bar{h}(x,\mu,z)|^l\nu(dz)ds=\frac{1}{t}\int_{0}^{t}\phi^l(x)[1-(1-e^{-s})]^lds
=\phi^l(x)\frac{1-e^{-lt}}{lt}
\leqslant \varphi(t)C_R^{h}(1+|x|^l),\;\;l=r \text{ or }\kappa,\notag
\end{align}
for $x\in\mathbb{R}$ with $|x|\leqslant R$, where $\varphi_1(t)=\frac{1}{1+t}$, $\varphi_2(t)=\frac{1}{1+t}$, $\varphi_3(t)=\frac{1-e^{-2t}}{2t}$ and $\varphi(t)=\frac{1-e^{-lt}}{lt}$ are continuous and positive bounded functions with $\lim_{t\to\infty}\varphi_i(t)=\lim_{t\to\infty}\varphi(t)=0,\;i=1,2,3$.

Based on the above discussion and the result of Theorem \ref{Th}, the solution of equation (\ref{Example-equation}) can be approximated by the following equation (with $\bar{X}(0)=x_0$)
\begin{align}\label{AExample-equation}
d\bar{X}(t)=&\left[\left(\bar{X}(t)-\bar{X}^3(t)\right)+\mathbb{E}\bar{X}(t)\right]dt+\left[\bar{X}(t)\sin\left(\log^2\left(1+\bar{X}^2(t)\right)\right)+\mathbb{E}\bar{X}(t)\right]dW(t) \notag\\
&+\int_{U}\left[\bar{X}(t)\sin\left(\log^{\frac{3}{2}}\left(1+\bar{X}^2(t)\right)\right)+\mathbb{E}\bar{X}(t)\right]\tilde{N}(dt,dz), \;\;\; t\in[0,T],
\end{align}
in the sense of mean square.

Now we carry out the numerical simulation to compute the solutions of  (\ref{Example-equation}) and  (\ref{AExample-equation}) under conditions $x_0=1$, $T=10$, $\varepsilon=0.01$ and $x_0=1$, $T=10$, $\varepsilon=0.001$, respectively. Fig. \ref{1a} and Fig.\ref{1b} show the comparisons of the exact solution $X_{\varepsilon}(t)$ for (\ref{Example-equation}) with the averaged solution $\bar{X}(t)$ for (\ref{AExample-equation}). As we can see, there is a good agreement between solutions of the original equation and the averaged equation. One can further find that error $X_{\varepsilon}(t)-\bar{X}(t)$ decreases  when $\varepsilon$ varies from $0.01$ to $0.001$. This result agrees well with our averaging principle (Theorem \ref{Th}).

\begin{figure}[h]
\centering
\vspace{-0.2cm}
\subfigure[$X_{\varepsilon}(0)=\bar{X}(0)=1$, $\varepsilon=0.01$]{\label{1a}
\includegraphics[width=0.86\textwidth]{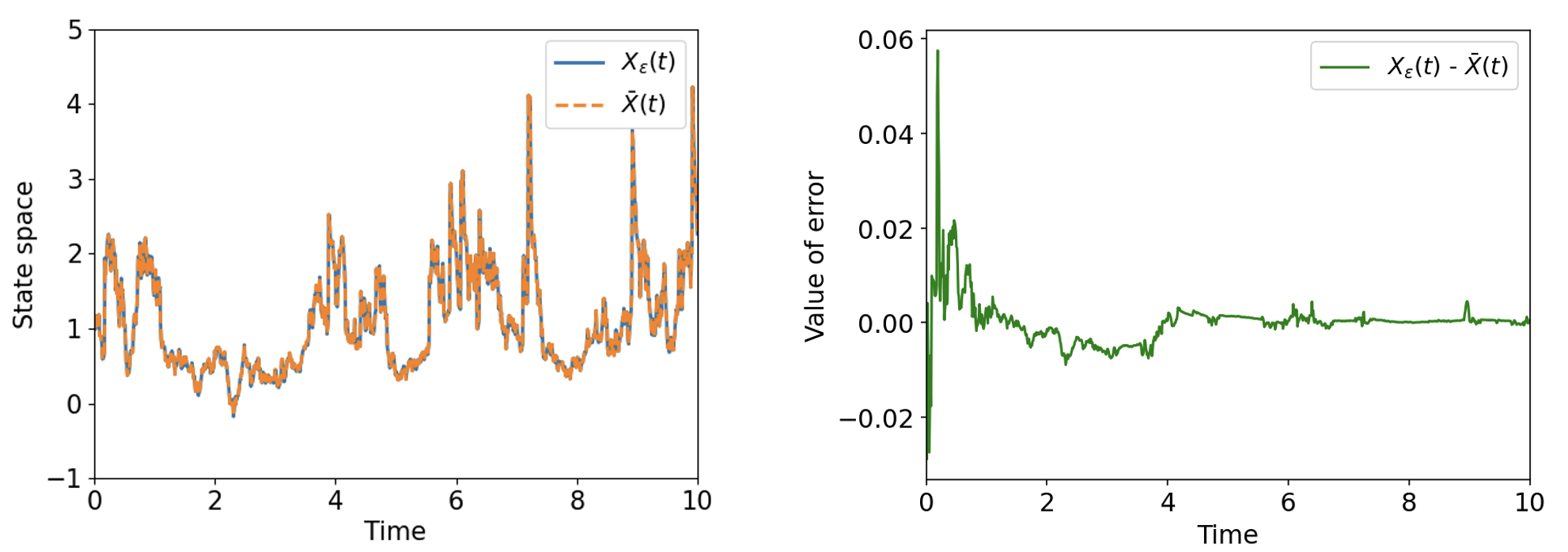}
}
\vspace{-0.2cm}
\subfigure[$X_{\varepsilon}(0)=\bar{X}(0)=1$, $\varepsilon=0.001$]{\label{1b}
\includegraphics[width=0.86\textwidth]{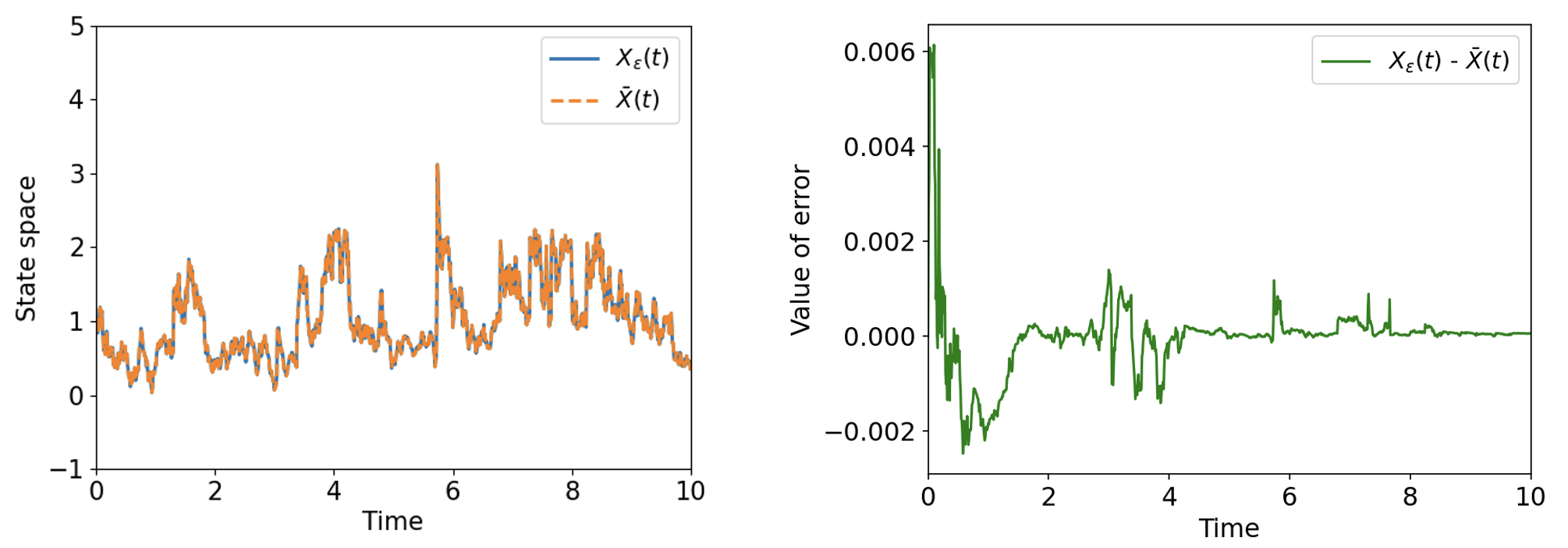}
}
\caption{(Color online) Comparison of the solution $X_{\varepsilon}(t)$ for (\ref{Example-equation}) with the averaged solution $\bar{X}(t)$ for (\ref{AExample-equation}). 
}
 \end{figure}

\begin{figure}[h]
 \centering
 \vspace{0cm}
\includegraphics[width=0.665\textwidth]{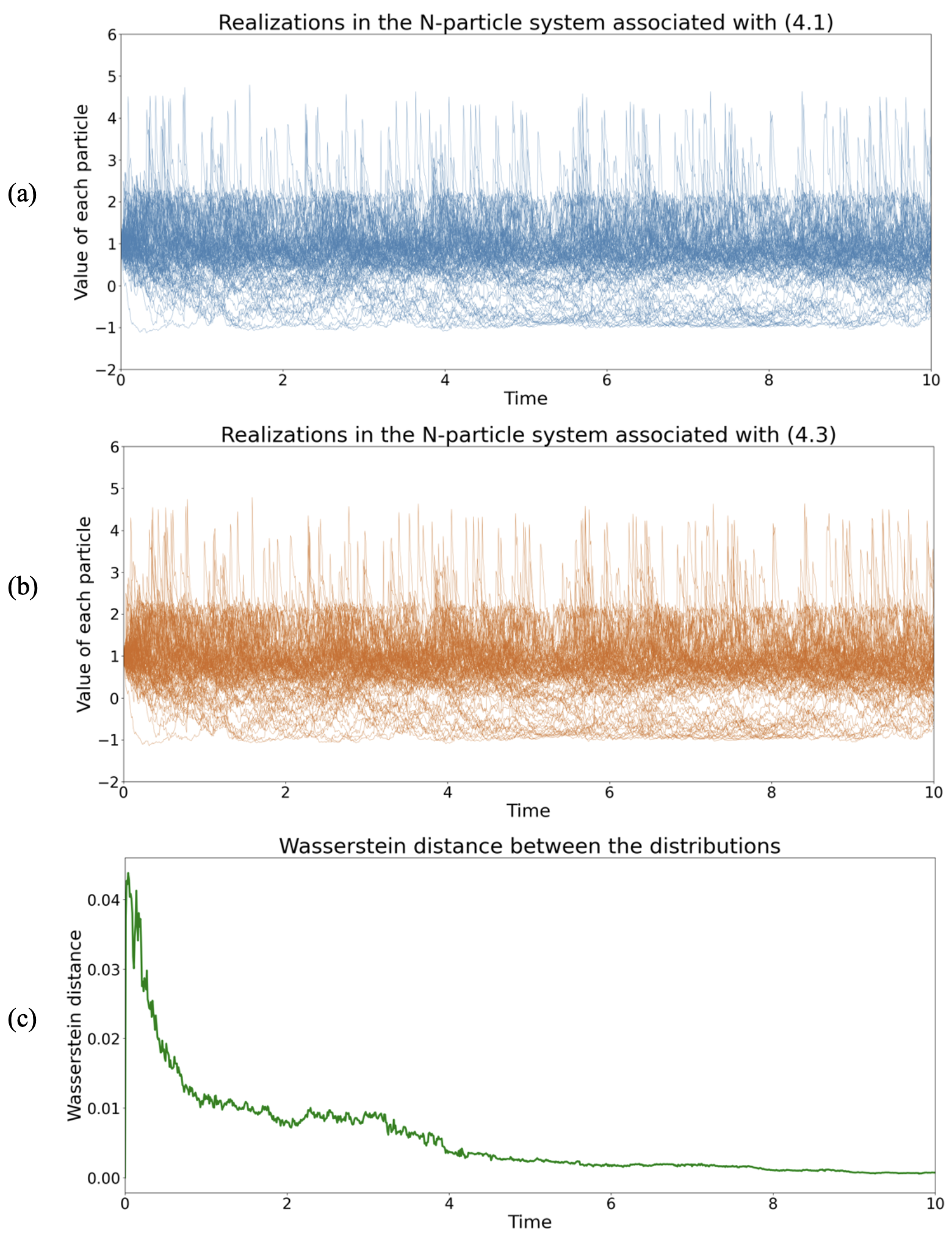}
\caption{(Color online) Comparison of the interacting N-particle systems associated with McKean-Vlasov SDEs (4.1) and (4.3), for $N=100$ and $\varepsilon=0.01$. 
 }
 \end{figure}

We remark that in numerical simulation, to approximate McKean-Vlasov SDE with (\ref{Example-equation}) and (\ref{AExample-equation}), we have used $N$-dimensional systems of interacting particles, respectively, which  can be
regarded as usual SDEs. This is so-called propagation of chaos result (see Appendix). Moreover, we note that, to simulate the integrals w.r.t. the compensated Poisson random measure $\tilde N(dt,dz) = N(dt,dz) - \nu (dz)dt$, we have applied the technique of introducing a compound Poisson process $\int_U z\tilde N(t,dz)$; See Section 4.3.2 in \cite{Applebaum2009}.

For our example, we simulate $N=100$ particles with a time step $0.01$, $T=10$. Fig.2(a) and Fig.2(b) show the realizations of the interacting particle systems associated with McKean-Vlasov SDE (\ref{Example-equation}) and (\ref{AExample-equation}), respectively, with the same initial conditions $X_{\varepsilon}(0)=\bar{X}(0)=1$ and $\varepsilon=0.01$. Numerically, the Wasserstein distance between distributions of (\ref{Example-equation}) (i.e., $\mathscr{L}_{X_{\varepsilon}(t)}$)  and (\ref{AExample-equation}) (i.e., $\mathscr{L}_{\bar{X}(t)}$ ) is thus approximated by those of the interacting particle systems; See Fig.2(c).

\section*{Acknowledgements}
This research was supported by the National Natural Science Foundation of China (12101484, 12141107), the Fundamental Research Funds for the Central Universities (xzy012022001, 5003011053), and the China Postdoctoral Science Foundation (2022TQ0009, 2022M720264).

\bibliographystyle{plain}
\bibliography{yingreferences}

\section*{Appendix}

\begin{proof}[\bf A. The details of the proof for Lemma \ref{lemma3-1}]
For $\forall t\in[0,T]$, $\forall x,y\in\mathbb{R}^d$ and  $\forall \mu,\mu_1,\mu_2 \in\mathcal{M}_{2}(\mathbb{R}^d)$, we calculate successively that
\begin{align}
    \left\langle x-y, \bar{b}(x,\mu)-\bar{b}(y,\mu)\right\rangle 
\leqslant&|x-y|^2+\frac{1}{2t}\int_0^t|b(s,x,\mu)-\bar{b}(x,\mu)|^2+|b(s,y,\mu)-\bar{b}(y,\mu)|^2ds+L_R|x-y|^2\notag\\
\leqslant&\varphi_1(t)C_R^b(1+|x|^2+|y|^2)+(\frac{1}{2}+L_R)|x-y|^2,\notag\\
\left\langle x, \bar{b}(x,\mu)\right\rangle 
\leqslant&\frac{|x|^2}{2}+\frac{1}{2t}\int_0^t|b(s,x,\mu)-\bar{b}(x,\mu)|^2ds+K(1+|x|^2+W_2^2(\mu,\delta_0))\notag\\
\leqslant&\frac{1}{2}\varphi_1(t)C_R^b(1+|x|^2)+\frac{3}{2}K(1+|x|^2+W_2^2(\mu,\delta_0)),\notag\\
|\bar{b}(x,\mu_1)-\bar{b}(x,\mu_2)|^2
\leqslant&\frac{3}{t}\int_0^t|b(s,x,\mu_1)-\bar{b}(x,\mu_1)|^2+|b(s,x,\mu_2)-\bar{b}(x,\mu_2)|^2+|b(s,x,\mu_1)-b(s,x,\mu_2)|^2ds\notag\\
\leqslant&6C_R^b\varphi_1(t)(1+|x|^2)+3LW_2^2(\mu_1,\mu_2),\notag\\
|\bar{b}(x,\mu_1)-\bar{b}(y,\mu_2)|^2
\leqslant&\frac{3}{t}\int_0^t|b(s,x,\mu_1)-\bar{b}(x,\mu_1)|^2+|b(s,y,\mu_2)-\bar{b}(y,\mu_2)|^2+|b(s,x,\mu_1)-b(s,y,\mu_2)|^2ds\notag\\
\leqslant&6C_R^b\varphi_1(t)(1+|x|^2)+\frac{3}{t}\int_0^t|b(s,x,\mu_1)-b(s,y,\mu_2)|^2ds,\notag\\
|\bar{b}(x,\mu)|^2
\leqslant&
2\Big|\frac{1}{t}\int_0^t b(s,x,\mu)-\bar{b}(x,\mu)ds\Big|^2+2\Big|\frac{1}{t}\int_0^tb(s,x,\mu)ds\Big|^2\notag\\
\leqslant& 2C_R^b\varphi_1(t)(1+|x|^2)+2K(1+|x|^{\kappa}+W_2^{\kappa}(\mu,\delta_0)).\notag
\end{align}
Similar results for $\bar{\sigma}$ and $\bar{h}$ are also available. Let $t\to\infty$ in above estimates, we conclude that the averaged equation (\ref{ASDE}) satisfies Assumptions {\bf A1-A5}.

We next check the extra conditions for $\bar{h}$ by calculating that (for $l=r$ or $\kappa$)
\begin{align}
\int_U|\bar{h}(x,\mu,z)|^l\nu(dz)
&\leqslant\frac{2^{l-1}}{t}\int_0^t \int_U|h(s,x,\mu,z)-\bar{h}(x,\mu,z)|^r+|h(s,x,\mu,z)|^r\nu(dz)ds\notag\\
&\leqslant 2^{l-1}C_R^{h}\varphi(t)(1+|x|^l)+2^{l-1}K(1+|x|^{l}+W_2^{l}(\mu,\delta_0)),\notag\\
\int_U|\bar{h}(x,\mu,z)-\bar{h}(y,\mu,z)|^{\kappa}\nu(dz)
\leqslant&\frac{3^{\kappa-1}}{t}\int_0^t\int_U|h(s,x,\mu,z)-\bar{h}(x,\mu,z)|^{\kappa}+|h(s,y,\mu,z)-\bar{h}(y,\mu,z)|^{\kappa}\notag\\
&+|h(s,x,\mu,z)-h(s,y,\mu,z)|^{\kappa}\nu(dz)ds\notag\\
\leqslant&
3^{\kappa-1}\varphi(t)C_R^{h}(2+|x|^{\kappa}+|y|^{\kappa})+3^{\kappa-1}L_R^{'}|x-y|^{\kappa},
\notag\\
\int_U|\bar{h}(x,\mu_1,z)-\bar{h}(x,\mu_2,z)|^{\kappa}\nu(dz)
\leqslant&
\frac{3^{\kappa-1}}{t}\int_0^t\int_U|h(s,x,\mu_1,z)-\bar{h}(x,\mu_1,z)|^{\kappa}+|h(s,x,\mu_2,z)-\bar{h}(x,\mu_2,z)|^{\kappa}\notag\\
&+\int_U|h(s,x,\mu_1,z)-h(s,x,\mu_2,z)|^{\kappa}\nu(dz)ds\notag\\
\leqslant&2\cdot 3^{\kappa-1}\varphi(t)C_R^{h}(1+|x|^{\kappa})+3^{\kappa-1}L^{'}W_2^{\kappa}(\mu_1,\mu_2),\notag
\end{align}
and then taking $t\to\infty$. That is, Assumption {\bf A7} holds.
\end{proof}

\noindent{\bf B. Propagation of chaos.}
For $N\geqslant1$ and $i=1,2,\ldots, N$, let $(W^i,\tilde{N}^i,X^i(0))$ be independent copies of $(W,\tilde{N},X(0))$. We introduce the non-interacting particle system associated with the McKean-Vlasov SDE (\ref{Main-equation}):
\begin{align}\label{Non-Main-equation}
&dX^i(t)=b\left(t,X^i(t),\mathscr{L}_{X^i(t)}\right) dt+\sigma\left(t,X^i(t),\mathscr{L}_{X^i(t)}\right)dW^i(t)+\int_{U}h\left(t,X^i(t),\mathscr{L}_{X^i(t)},z\right)\tilde{N}^i(dt,dz)
\end{align}
on $t\in[0,T]$ with initial data $X^i(0)$. According to Theorem \ref{mainresult1}, one has $\mathscr{L}_{X^i(t)}=\mathscr{L}_{X(t)}$, $i=1,2,\ldots, N$.

We also consider the associated interacting particle system
\begin{align}\label{Inter-Main-equation}
&dX^{i,N}(t)=b\left(t,X^{i,N}(t),\mu_t^{X,N}\right) dt+\sigma\left(t,X^{i,N}(t),\mu_t^{X,N}\right)dW^{i}(t)+\int_{U}h\left(t,X^{i,N}(t),\mu_t^{X,N},z\right)\tilde{N}^{i}(dt,dz)
\end{align}
with initial data $X^{i,N}(0)=X^i(0)$, where $\mu_t^{X,N}$ is an empirical measure of $N$ particles given by $\mu_t^{X,N}:=\frac{1}{N}\sum_{j=1}^N\delta_{X^{j,N}(t)}$. We next state and prove the propagation of chaos result.
\begin{proposition}(Propagation of chaos) Suppose Assumptions {\bf A1-A7} hold and $r\geqslant4$. Then, the interacting particle system (\ref{Inter-Main-equation}) is well-posed and converges to the non-interacting particle system (\ref{Non-Main-equation}), that is,
\begin{equation}\label{Poc}
    \lim_{N\to\infty}\sup_{1\leqslant i\leqslant N}\sup_{0\leqslant t\leqslant T}\mathbb{E}\left|X^i(t)-X^{i,N}(t)\right|^2=0.
\end{equation}
\end{proposition}

\begin{proof}
First, note that the interacting particle system $\{X^{i,N}\}_{1\leqslant i\leqslant N}$ given in equation (\ref{Inter-Main-equation}) can be regarded as ordinary SDEs driven by L\'evy noise taking values in $\mathbb{R}^{d\times N}$. Thus, according to the Theorem 1.1 of \cite{Majka2016note}, it has a unique c\`{a}dl\`{a}g solution under Assumptions {\bf A1, A4, A5} such that 
$$\sup_{1\leqslant i\leqslant N}\mathbb{E}\sup_{0\leqslant t\leqslant T}\left|X^{i,N}(t)\right|^4\leqslant C,$$
for any $N\geqslant1$ where $C>0$ is independent of $N$.

To treat the one-sided local Lipschitz case, for any $1\leqslant i\leqslant N$ and $R>0$, define the stopping time:
$$\zeta_R:=\inf\left\{t\in[0,T]: |X^i(t)|\vee|X^{i,N}(t)|> R\right\}.$$
Then, by De Morgan's Law, we arrive at
\begin{equation*}
\sup_{0\leqslant t\leqslant T}\mathbb{E}\left|X^i(t)-X^{i,N}(t)\right|^{2}\leqslant \sup_{0\leqslant t\leqslant T}\mathbb{E}\left[\left|X^i(t)-X^{i,N}(t)\right|^{2}\mathbb{I}_{\{\zeta_R>T\}}\right]+\sup_{0\leqslant t\leqslant T}\mathbb{E}\left[\left|X^i(t)-X^{i,N}(t)\right|^{2}\mathbb{I}_{\{\zeta_R\leqslant T\}}\right]:=Q_1+Q_2,
\end{equation*}
where $\mathbb{I}_A$ is an indicator function of set $A$. Similar to the proof of Theorem \ref{Th}, we now calculate $Q_1$ and $Q_2$ respectively.

(1) Estimation of the term $Q_1$. Note that 
\begin{equation}\label{Q_1}
Q_1=\sup_{0\leqslant t\leqslant T}\mathbb{E}\left[\left|X^i(t)-X^{i,N}(t)\right|^{2}\mathbb{I}_{\{\zeta_R>T\}}\right]\leqslant \sup_{0\leqslant t\leqslant T}\mathbb{E}\left|X^i(t\wedge\zeta_R)-X^{i,N}(t\wedge\zeta_R)\right|^{2}.
\end{equation}
By It\^o formula, 
\begin{align}
\mathbb{E}\left|X^i(t\wedge\zeta_R)-X^{i,N}(t\wedge\zeta_R)\right|^{2}
=&2\mathbb{E}\int_0^{t\wedge\zeta_R}\left\langle X^i(s)-X^{i,N}(s), b\left(s,X^i(s),\mathscr{L}_{X^i(s)}\right)-b\left(s,X^{i,N}(s),\mu_s^{X,N}\right)\right\rangle ds\notag\\
&+\mathbb{E}\int_0^{t\wedge\zeta_R}\left\| \sigma\left(s,X^i(s),\mathscr{L}_{X^i(s)}\right)-\sigma\left(s,X^{i,N}(s),\mu_s^{X,N}\right) \right\|^2 ds\notag\\
&+\mathbb{E}\int_0^{t\wedge\zeta_R}\int_U\left|h\left(s,X^i(s),\mathscr{L}_{X^i(s)},z\right)-h\left(s,X^{i,N}(s),\mu_s^{X,N},z\right)\right|^2 \nu(dz)ds:=\sum_{i=1}^3Q_{i,R}(t).\notag 
\end{align} 
We estimate $Q_{i,R}$, $i=1,2,3$ by Assumptions {\bf A1-A2} and obtain that
\begin{align}
Q_{1,R}(t)
\leqslant&(2L_R+1)\int_0^{t\wedge\zeta_R}\mathbb{E}\left|X^i(s)-X^{i,N}(s)\right|^2ds+L\int_0^{t\wedge\zeta_R} \mathbb{E}W^2\left(\mathscr{L}_{X^{i}(s)},\mu_s^{X,N}\right)ds,\notag\\
Q_{2,R}(t)
\leqslant&2L_R\int_0^{t\wedge\zeta_R}\mathbb{E}\left|X^i(s)-X^{i,N}(s)\right|^2ds+L\int_0^{t\wedge\zeta_R}\mathbb{E}W_2^2\left(\mathscr{L}_{X^{i}(s)},\mu_s^{X,N}\right)ds,\notag\\
Q_{3,R}(t)\leqslant
&2L_R\int_0^{t\wedge\zeta_R}\mathbb{E}\left|X^i(s)-X^{i,N}(s)\right|^2ds+L\int_0^{t\wedge\zeta_R}\mathbb{E}W_2^2\left(\mathscr{L}_{X^{i}(s)},\mu_s^{X,N}\right)ds,\notag
\end{align}
where
$$
\int_0^{t\wedge\zeta_R}\mathbb{E}W_2^2\left(\mathscr{L}_{X^{i}(s)},\mu_s^{X,N}\right)ds \leqslant
2\int_0^{t\wedge\zeta_R}\mathbb{E}W_2^2\left(\mathscr{L}_{X^{i}(s)},\mu_s^{X}\right)ds+2\int_0^{t\wedge\zeta_R}\mathbb{E}W_2^2\left(\mu_s^{X},\mu_s^{X,N}\right)ds,
$$
with $\mu_t^X:=\frac{1}{N}\sum_{i=1}^N\delta_{X^{i}(t)}$ the empirical measure. Then, by combining these estimates and applying the Gr\"onwall inequality, we eventually have
\begin{align}
Q_1\leqslant\sup_{0\leqslant t\leqslant T}\mathbb{E}\left|X^i(t\wedge\zeta_R)-X^{i,N}(t\wedge\zeta_R)\right|^{2}
\leqslant 6Le^{(6L_R+1+6L)T}\int_0^{T}\mathbb{E}W_2^2\left(\mathscr{L}_{X^{i}(s)},\mu_s^{X}\right)ds.\notag
\end{align} 

(2) Estimation of the term $Q_2$.  Using the Cauchy-Schwarz inequality and Theorem \ref{mainresult1}, we deduce that
\begin{align}\label{Q-2}
Q_2=\sup_{0\leqslant t\leqslant T}\mathbb{E}\left[\left|X^i(t)-X^{i,N}(t)\right|^2\mathbb{I}_{\{\zeta_R\leqslant T\}}\right]
\leqslant&\sup_{0\leqslant t\leqslant T}\sqrt{\mathbb{E}\left(\left|X^i(t)-X^{i,N}(t)\right|^2\right)^2}\sqrt{\mathbb{E}\left(\mathbb{I}_{\{\zeta_R\leqslant T\}}\right)^2}\notag\\
\leqslant&\frac{2\sqrt{2}}{R^2}\left(\mathbb{E}\sup_{0\leqslant t\leqslant T}|X^i(t)|^4+\mathbb{E}\sup_{0\leqslant t\leqslant T}|{X}^{i,N}(t)|^4\right)\leqslant \frac{C}{R^2}.
\end{align}

With the estimations of $Q_1$ and $Q_2$ at hand, we conclude that 
\begin{align}\label{Q-1+2}
\sup_{1\leqslant i\leqslant N}\sup_{0\leqslant t\leqslant T}\mathbb{E}\left|X^i(t)-X^{i,N}(t)\right|^2\leqslant6Le^{(6L_R+1+6L)T}\int_0^{T}\sup_{1\leqslant i\leqslant N}\mathbb{E}W_2^2\left(\mathscr{L}_{X^{i}(s)},\mu_s^{X}\right)ds+\frac{C}{R^2}.
\end{align}
Notice that, by the Theorem 5.8 of \cite{Carmona2018probabilistic},
 \begin{equation*}   
  \begin{split}
\mathbb{E}W_2^2\left(\mathscr{L}_{X^{i}(s)},\mu_s^{X}\right)\leqslant C \left \{
  \begin{array}{lll}
   N^{-1/2},& \text{if } d<4,\\
   N^{-1/2}\ln(N),&\text{if } d=4,\\
   N^{-1/2},&\text{if } d>4.
  \end{array}
   \right.
    \end{split}
  \end{equation*}
We find that the right hand of the estimate \eqref{Q-1+2} converges to 0 when $N$ goes to infinity. The result thus follows and the proof is completed.
\end{proof}

\end{document}